\documentclass{article}

\usepackage{amsmath}
\usepackage{amsthm}
\usepackage{amsfonts}
\usepackage{amssymb}
\usepackage{mathrsfs}
\usepackage{amsmath}
\usepackage{verbatim}
\usepackage{manfnt}
\usepackage{xypic}
\usepackage{comment}
\usepackage{hyperref}
\usepackage{color}

\theoremstyle{definition}
\newtheorem {definition} {Definition}[section]
\newtheorem {example}[definition]{Example}

\theoremstyle{theorem}
\newtheorem {lemma} [definition]  {Lemma}

\newtheorem {proposition}  [definition]  {Proposition}
\newtheorem {corollary} [definition] {Corollary}

\theoremstyle{theorem}
\newtheorem{mainthm}{Theorem}[section]

\newtheorem{mainlemma}{Lemma}[section]

\theoremstyle{remark}
\newtheorem {remark}  [definition]  {Remark}

\begin{document}

\title{K-theory of affine actions}
\author{James Waldron}
\date{12/10/2017}

\maketitle

\abstract{
For a Lie group $G$ and a vector bundle $E$ we study those actions of the Lie group $TG$ on $E$ for which the action map $TG\times E \to E$ is a morphism of vector bundles, and call those \emph{affine actions}. We prove that the category $\mathrm{Vect}_{TG}^{\mathrm{aff}}\left(X\right)$ of such actions over a fixed $G$-manifold $X$ is equivalent to a certain slice category $\mathfrak g_X \backslash \mathrm{Vect}_G\left(X\right)$. We show that there is a monadic adjunction relating $\mathrm{Vect}_{TG}^{\mathrm{aff}}\left(X\right)$ to $\mathrm{Vect}_G\left(X\right)$, and the right adjoint of this adjunction induces an isomorphism of Grothendieck groups $K_{TG}^{\mathrm{aff}}\left(X\right) \cong KO_G\left(X\right)$. Complexification produces analogous results involving $T_\mathbb C G$ and $K_G\left(X\right)$.
}

\section{Introduction}

\subsection{Introduction}

Let $G$ be a Lie group. The tangent bundle $TG$ carries two structures: it is a vector bundle over $G$, and a Lie group, with multiplication given by the derivative of the multiplication of $G$. These structures are compatible, in the sense that the multiplication $TG\times TG \to TG$ is a morphism of vector bundles, so that $TG$ is a \emph{group object} in the category of vector bundles. It is therefore natural to study actions of $TG$ on vector bundles, such that the action map $TG \times E \to E$ is a morphism of vector bundles (see Definition \ref{def: affine actions} below). We refer to such actions as \emph{affine actions}, as each element of $TG$ necessarily acts by an affine linear transformation between fibres of $E$ (see Remark \ref{rem: affine} below). A basic example of an affine action is the following:
\begin{example}\textbf{Tangent bundles.}
\label{ex: intro tangent}
For any action $t:G\times X \to X$ of a Lie group $G$ on a smooth manifold $X$, the derivative defines an affine action $t_\ast : TG\times TX \to TX$ of $TG$ on $TX$. Note that restricting the action $t_\ast$ to $G$ defines the natural action $G\times TX \to TX$ of $G$ on $TX$, whilst restricting $t_\ast$ to the Lie algebra $\mathfrak g=T_eG$ allows one to define a map $\mathfrak g \to \Gamma\left(TX\right)$ which is exactly the infinitesimal action associated to $t$. These maps are compatible in the sense that $\mathfrak g \to \Gamma\left(TX\right)$ is $G$-equivariant. 
\end{example}

Example \ref{ex: intro tangent} suggests the question of whether the action $t_\ast :TG\times TX \to TX$, or more generally any affine action $\mu:TG\times E \to E$, can be reconstructed from its restrictions to $G$ and $\mathfrak g$.

One motivation for studying equivariant vector bundles is their use in defining the (real) equivariant K-theory $KO_G\left(X\right)$ of a $G$-manifold $X$ (see \cite{segal}). Recall that (at least if $G$ and $X$ are compact) $KO_G\left(X\right)$ is the Grothendieck group of the commutative monoid of isomorphism classes of $G$-equivariant real vector bundles over $X$.  A natural question to ask is whether one can emulate this construction in the case of affine actions to define an abelian group $K_{TG}^{\mathrm{aff}}\left(X\right)$. If so, how is this group related to $KO_G\left(X\right)$?

A different motivation for studying affine actions comes from the theory of \emph{Lie algebroids}. Recall that a Lie algebroid $A\to X$ is a vector bundle over $X$ equipped with an $\mathbb R$-linear Lie bracket on $\Gamma\left(A\right)$ and a map $A\to TX$, such that a certain Leibniz rule is satisfied. (See \cite{mackenzie} for more details.) There exists a notion of equivariant Lie algebroid (called a \emph{Harish-Chandra Lie algebroid} in \cite{bb}) which involves both a $G$-action $G\times A \to A$ and a linear map $\mathfrak g \to \Gamma\left(A\right)$ satisfying certain conditions. Variants of this notion have appeared in \cite{alekseev},\cite{bruzzo},\cite{marrero},\cite{ginzburg}. It was shown in \cite{marrero} that equivariant Lie algebroids give rise to examples of affine actions (see Example \ref{ex: Liealg} below). One motivation for our results is therefore to generalise the notion of affine action and study this concept at the level of vector bundles.

\subsection{Main results}

Throughout the paper $G$ is a real Lie group and $X$ is a $G$-manifold. See section \ref{sec: prelim} for our notation and conventions and section \ref{def: affine actions} for the definition of affine actions and their morphisms. We use $\mathrm{Vect}_{TG}^{\mathrm{aff}}\left(X\right)$ (respectively $\mathrm{Vect}_G\left(X\right)$) to denote the category of affine actions (respectively the category of real $G$-equivariant vector bundles) over $X$. We use $\mathfrak g_X$ to denote the $G$-equivariant vector bundle associated to the adjoint representation of $G$. See section \ref{subsec: the category} for the definition of the slice category $\mathfrak g_X \backslash \mathrm{Vect}_G\left(X\right)$.

\begingroup
\setcounter{mainthm}{0} 
\renewcommand\themainthm{\Alph{mainthm}}
\begin{mainthm}
\label{thm: A}
The following three categories are isomorphic:
\begin{enumerate}
\item The category $\mathrm{Vect}_{TG}^\mathrm{aff}\left(X\right)$ of affine actions of $TG$ over $X$.
\item The category of pairs $\left(E,\rho_\mathfrak g\right)$, defined as follows: 
\begin{itemize}
\item The objects are pairs $\left(E,\rho_\mathfrak g\right)$, where $E$ is a $G$-equivariant vector bundle over $X$ and $\rho_\mathfrak g : \mathfrak g \to \Gamma\left(E\right)$ is a $G$-equivariant linear map. 
\item The morphisms $\left(E,\rho_\mathfrak g\right) \to \left(E',\rho_\mathfrak g\right)$ are morphisms $\psi:E\to E'$ of $G$-equivariant vector bundles over $X$ such that $\Gamma\left(\psi\right)\circ \rho_\mathfrak g = \rho'_\mathfrak g$.
\item Composition is given by composition of morphisms of vector bundles over $X$.
\end{itemize}
\item The slice category $ \mathfrak g_X \backslash \mathrm{Vect}_G\left(X\right)$.
\end{enumerate}
\end{mainthm}
\endgroup

Our second main result compares affine actions to equivariant vector bundles. There is a canonical forgetful functor 
\[
\mathrm U:\mathrm{Vect}_{TG}^{\mathrm{aff}}\left(X\right) \to \mathrm{Vect}_G\left(X\right)
\] 
Via the isomorphism $\mathrm{Vect}_{TG}^{\mathrm{aff}}\left(X\right) \cong \mathfrak g_X \backslash\mathrm{Vect}_G\left(X\right)$ of Theorem \ref{thm: A}, the functor $\mathrm{U}$ is equal to the canonical forgetful functor 
$\mathfrak g_X \backslash \mathrm{Vect}_G\left(X\right) \to \mathrm{Vect}_G\left(X\right)$
which maps an object $\mathfrak g_X \xrightarrow{\phi} E$ to $E$. We also define a pair of functors $\mathrm F,\sigma:\mathrm{Vect}_G\left(X\right) \to \mathrm{Vect}_{TG}^{\mathrm{aff}}\left(X\right)$.
In terms of $\mathfrak g \backslash \mathrm{Vect}_G\left(X\right)$, they are defined on objects by $\sigma: E \mapsto \left(E,0\right)$ and $\mathrm F: E \mapsto \left(\mathfrak g_X\oplus E,i_{\mathfrak g_X}\right)$. See section \ref{subsec: forget} for the precise definitions.
\begingroup
\setcounter{mainthm}{1} 
\renewcommand\themainthm{\Alph{mainthm}}
\begin{mainthm}
\label{thm: B}
The following statements hold:
\begin{enumerate}
\item $\mathrm F$ is left adjoint to $\mathrm U$.
\item The adjunction $\mathrm F \dashv \mathrm U$  is monadic.
\item $\sigma$ is the unique section of $\mathrm U$.
\end{enumerate}
\end{mainthm}
\endgroup

Our third main result concerns the Grothendieck group of $\mathrm{Vect}_{TG}^{\mathrm{aff}}\left(X\right)$. We denote by 
\[
KO_G\left(-\right) : G\text{-Man} \to \mathrm{Ab}
\]
the functor from the category $G\text{-Man}$ of $G$-manifolds to the category $\mathrm{Ab}$ of abelian groups, which maps a $G$-manifold $X$ to the Grothendieck group of $G$-equivariant real vector bundles over $X$. (This agrees with real $G$-equivariant topological $K$-theory as defined by Segal in \cite{segal} if both $G$ and $X$ are compact.)
Although the category $\mathrm{Vect}_{TG}^{\mathrm{aff}}\left(X\right)$ is not additive we show in section \ref{subsec: products} that it does have finite products. This allows us to define the Grothendieck group $K_{TG}^{\mathrm{aff}}\left(X\right)$ (see section \ref{subsec: Grothendieck} and Definition \ref{def: KTG}). This construction extends to a contravariant functor
\[
K_{TG}^{\mathrm{aff}}\left(-\right) : G\text{-Man} \to \mathrm{Ab}
\]
Our third main result shows that $K_{TG}^{\mathrm{aff}}\left(-\right)$ agrees with $KO_G\left(-\right)$:

\begingroup
\setcounter{mainthm}{2} 
\renewcommand\themainthm{\Alph{mainthm}}
\begin{mainthm}
\label{thm: C}
For $X$ a $G$-manifold the functor $\mathrm U$ induces a group isomorphism
\[
K\left(\mathrm U\right):K_{TG}^{\mathrm{aff}}\left(X\right) \to KO_G\left(X\right)
\]
Its inverse is 
\[
K\left(\sigma\right) : KO_G\left(X\right) \to K_{TG}^{\mathrm{aff}}\left(X\right)
\]
These isomorphisms are natural in $X$, and thus define an isomorphism of functors
\[
K_{TG}^{\mathrm{aff}}\left(-\right) \xrightarrow{\cong} KO_G\left(-\right)
\]
\end{mainthm}
\endgroup

\subsection{The complex case}
\label{subsec: complex case}

It is possible to reformulate the notion of affine action in the complex setting by replacing $TG$ by the complexified tangent bundle $T_\mathbb C G$ and considering actions $T_\mathbb CG\times E \to E$ on complex vector bundles $E$. The analogues of Theorems \ref{thm: A},\ref{thm: B} and \ref{thm: C} hold with essentially the same proofs. See section \ref{sec: complex} for the precise statements.

\subsection{The proofs}

Theorem A is proved using the facts that $TG \cong G \times \mathfrak g$ as a vector bundle, and $TG \cong \mathfrak g \rtimes G$ as a Lie group, where the semi-direct product is defined via the adjoint representation of $G$. This allows one to decompose an action $\mu$ of $TG$ into an action $\mu_G$ of $G$ and a linear map $\rho_\mathfrak g$ with domain $\mathfrak g$. Parts of the proof are similar to that of Theorem 3.5 in \cite{marrero}, which deals with the particular case where $E$ is a Lie algebroid, and constructs, at the level of objects, one direction of the isomorphism of Theorem \ref{thm: A}. 

Using the isomorphism $\mathrm{Vect}_{TG}^{\mathrm{aff}}\left(X\right) \cong  \mathfrak g_X \backslash \mathrm{Vect}_G\left(X\right)$ of Theorem A, Theorems B and C are proved using the following two category-theoretic Lemmas regarding over slice categories:

\begingroup
\setcounter{mainlemma}{3} 
\renewcommand\themainlemma{\Alph{mainlemma}}
\begin{mainlemma}
\label{lemma: D}
Let $\mathscr C$ be a category, $m$ an object in $\mathscr C$, and $\mathcal U: m\backslash \mathscr C \to \mathscr C$ the standard forgetful functor. If the coproduct $m \amalg a$ exists in $\mathscr C$ for all objects $a$ in $\mathscr C$ then the functor $\mathcal F: a \mapsto \left(m\amalg a, i_m\right)$ is left adjoint to $\mathcal U$, and this adjunction is monadic.
\end{mainlemma}
\endgroup

\begingroup
\setcounter{mainlemma}{4} 
\renewcommand\themainlemma{\Alph{mainlemma}}
\begin{mainlemma}
\label{lemma: E}
Let $\mathscr C$ be an additive category and $m$ an object in $\mathscr C$. Let $\mathcal U: m\backslash \mathscr C \to \mathscr C$ be the standard forgetful functor and $\mathcal S: \mathscr C \to m \backslash \mathscr C$ the section $a \mapsto \left(a,0\right)$. Then the group homomorphism
\[
K\left(\mathcal U \right): K\left(m\backslash \mathscr C\right) \to K\left(\mathscr C\right) 
\]
is an isomorphism. Its inverse is
\[
K\left(\mathcal S \right) : 
K\left(\mathscr C\right) \to
 K\left(m\backslash \mathscr C\right) 
\]
\end{mainlemma}
\endgroup

Here, $K\left(\mathcal U\right)$ denotes the homomorphism of Grothendieck groups associated to the product preserving functor $\mathcal U$, and similarly for $K\left(\mathcal S\right)$, see section \ref{subsec: Grothendieck}. We expect that Lemma \ref{lemma: D} is well known to experts, (in particular it is stated without proof in \cite{nlab}), but we are unaware of a complete reference and so have provided a proof.

\subsection{Outline}

In section \ref{sec: prelim} we fix notation and conventions. In section \ref{sec: affine} we define affine actions and morphisms between them. The main result of section \ref{sec: affine} is Theorem \ref{thm: A}, the proof of which is broken into Lemma \ref {lemma: decomp} and Propositions \ref{prop: decomp}, \ref{prop: decomp2} and \ref{prop: decomp  homs}. In section \ref{sec: examples} we describe a number of examples of affine actions, and describe the category $\mathrm{Vect}_{TG}\left(X\right)$ for certain classes of groups $G$ and $G$-manifolds $X$. In section \ref{sec: functors} we define several functors between $\mathrm{Vect}_{TG}\left(X\right)$ and $\mathrm{Vect}_G\left(X\right)$. We then prove Lemma \ref{lemma: D} which is used to prove Theorem \ref{thm: B}. In section \ref{sec: K} we define pullback functors for affine actions. We then prove that the category $\mathrm{Vect}_{TG}^{\mathrm{aff}}\left(X\right)$ has finite products and use this to define the abelian group $K_{TG}^{\mathrm{aff}}\left(X\right)$ and the functor $K_{TG}^{\mathrm{aff}}\left(-\right):G\text{-Man}\to \mathrm{Ab}$. We then prove Lemma \ref{lemma: E} from which we prove Theorem \ref{thm: C}.

\subsection{Acknowledgements}

The author would like to thank Martina Balagovi\'{c} for useful comments on the first draft of this paper, and Peter J\o rgensen for his continued guidance during the completion of this work. The author would also like to thank the Department of Mathematics at the University of Zagreb for their hospitality during part of the time in which this work was completed.


\section{Preliminaries}
\label{sec: prelim}

\subsection{Notation and conventions}

By `manifold' we shall always mean smooth finite dimensional real manifold. Maps between manifolds are assumed to be smooth. Unless stated otherwise, by `vector bundle' we mean finite dimensional real vector bundle. We will usually denote manifolds by $X$ or $Y$, vector bundles by $E$ or $F$, vector fields by $v$ or $w$, and sections of vector bundles by $\xi$ or $\nu$.  For a vector bundle $E$ over $X$ we use $\pi_E:E \to X$ to denote the bundle projection, and $0_E:X \to E$ to denote the zero section. We allow morphisms of vector bundles over different bases. If $E$ and $F$ are vector bundles over $X$, then by `morphism of vector bundles over $X$' we mean a vector bundle morphism $\phi:E\to F$ which satisfies $\pi_F \circ \phi = \pi_E$. If $\phi$ is a morphism of this type, then $\Gamma\left(\phi\right):\Gamma\left(E\right) \to \left(F\right)$ denotes the associated linear map.

For $E\to X$ a vector bundle, $x\in X$ and $\xi \in \Gamma\left(E\right)$, we use $\xi_x$ to denote $\xi$ evaluated at $x$. We use $e_x$ to denote an element of $E_x$ and $v_x$ to denote an element of $T_xX$. We denote the zero element of $E_x$ by $0_x$. For a morphism of vector bundles $\phi:E \to F$ over $X$ and $x\in X$ we denote by $\phi_x:E_x \to F_x$ the restriction of $\phi$.

If $E\to X$ and $F\to Y$ are vector bundles then $E\times F$ is a vector bundle over $X\times Y$ in a natural way, with fibre over $\left(x,y\right)$ canonically isomorphic to $E_x \oplus F_y$.

We denote hom-sets in a category $\mathscr C$ by $\mathrm{Hom}_\mathscr C\left(-,-\right)$. We reserve the unadorned $\mathrm{Hom}$ for morphisms of real vector spaces (i.e. linear maps). If $G$ is a Lie group then we use $\mathrm{Hom}_G$ for morphisms of representations of $G$ (i.e. $G$-equivariant linear maps). We denote identity morphisms by $1_a$ or $\mathrm{id}_a$. If $\mathscr C$ is an additive category then we denote any zero-morphism by $0$. If $a \times b$ is a product in a category $\mathscr C$ then we denote the associated projections by $\mathrm{pr}_a:a\times b \to a$ and $\mathrm{pr}_b:a\times b \to b$. Similarly, if $a \amalg b$ is a coproduct then we denote the associated inclusions by $i_a:a \to a\amalg b$ and $i_b:b \to a\amalg b$. If $a \amalg b$ is a coproduct and $f:a \to c$ and $g:b \to c$ are morphisms, then we denote by $\left(f,g\right):a\amalg b \to c$ the associated morphism. We use a similar notation for morphisms into products.

If $G$ is a Lie group then by a `$G$-manifold' we shall mean a smooth manifold $X$ equipped with a smooth left action $t:G\times X \to X$. We denote by $t_g=t\left(g,-\right):X\to X$ the diffeomorphism associated to $g\in G$, which we also denote by $x \mapsto g \cdot x$. If $X$ and $Y$ are $G$-manifolds then by a $G$-map $f:X\to Y$ we mean a $G$-equivariant smooth map. 

For a $G$-manifold $X$, by `$G$-equivariant vector bundle', or just `$G$-vector bundle', we mean a vector bundle $E\to X$ equipped with a left action $\mu_G : G\times E \to E$ which covers $t:G\times X \to X$ and which is fibrewise linear. A morphism $\phi:E\to F$ of $G$-equivariant vector bundles is a morphism of vector bundles over $X$ which is also $G$-equivariant. We denote by $\mathrm{Vect}_G\left(X\right)$ (respectively $\mathrm{Vect}_G^\mathbb C\left(X\right)$) the category of real (respectively complex) $G$-equivariant vector bundles over $X$. (See \cite{segal} for generalities on equivariant vector bundles.) $\mathrm{Vect}_G\left(X\right)$ and $\mathrm{Vect}_G^\mathbb C\left(X\right)$ are both additive categories. In particular, for $E$ and $F$ $G$-vector bundles the $G$-vector bundle $E\oplus F$ is both the product and the coproduct of $E$ and $F$. 

We denote the category of finite dimensional real (respectively complex) representations of $G$ by $\mathrm{Rep}\left(G\right)$ (respectively $\mathrm{Rep}_\mathbb C \left(G\right)$). If $V$ is such a representation then we denote by $V_X$ the associated $G$-vector bundle, i.e. the trivial bundle $X\times V$ with $G$-action $g \cdot \left(x,v\right) = \left(g\cdot x,g\cdot v\right)$. If $E$ is a $G$-equivariant vector bundle then we consider $\Gamma\left(E\right)$ as a $G$-representation with $G$-action $\left(g \cdot \xi\right)_x = g \cdot \xi_{g^{-1}\cdot x}$. If $V$ and $W$ are representations of $G$, and $E$ is a $G$-vector bundle, then there exist bijections
\begin{align}
\label{eqn: hom(VX,WX)}
\mathrm{Hom}_{\mathrm{Vect}_G\left(X\right)} \left(V_X,W_X\right)
& \xrightarrow{\cong}
C^\infty\left(X,\mathrm{Hom}\left(V,W\right)\right)^G \\
\nonumber
\phi & \mapsto \left( x \mapsto \phi_x\right)
\end{align}

\begin{align}
\label{eqn: hom(VX,E)}
\mathrm{Hom}_{\mathrm{Vect}_G\left(X\right)} \left(V_X, E\right)
& \xrightarrow{\cong}
\mathrm{Hom}_G\left(V,\Gamma\left(E\right)\right)
 \\
\nonumber
\phi & \mapsto \left(v \mapsto \left(x\mapsto \phi\left(x,v\right)\right) \right)
\end{align}

\subsection{Tangent groups}
\label{subsec: TG}

Let $G$ be a Lie group and $\mathfrak g=T_eG$ its Lie algebra. We will usually denote elements of $G$ by $g$ or $h$, and elements of $\mathfrak g$ by $\alpha$ or $\beta$. The tangent bundle $TG$ of $G$ carries a natural Lie group structure with multiplication defined by the composite map
$TG \times TG \xrightarrow{\cong}
T\left(G\times G\right) \xrightarrow{m_\ast} TG$ where $m:G\times G \to G$ is the multiplication of $G$. We will denote the multiplication in $TG$ by $\bullet$. If $v_g \in T_gG$ and $w_h \in T_hG$ then it follows from the chain rule that
\begin{equation*}
v_g \bullet w_h = \left(L_g\right)_\ast w_h + \left(R_h\right)_\ast v_g
\end{equation*}
In particular, $0_g \bullet \left(-\right) = \left(L_g\right)_\ast$, $\left(-\right)\bullet 0_h = \left(R_h\right)_\ast$,  $0_g \bullet 0_h = 0_{gh}$, and if $\alpha,\beta \in \mathfrak g$ then $\alpha \bullet \beta = \alpha + \beta$.
If one considers $\mathfrak g$ as an abelian Lie group upon which $G$ acts via the adjoint representation then the associated semi-direct product $\mathfrak g \rtimes G$ has multiplication 
\[
\left(\alpha,g\right) \bullet \left(\beta,h\right) 
= \left(\alpha+\mathrm{Ad}_g\beta,gh\right)
\] 
There is a Lie group isomorphism $\mathfrak g \rtimes G \xrightarrow{\cong} TG$ given by $\left(\alpha,g\right) \mapsto \left(R_g\right)_\ast \alpha$. Under this isomorphism, the inclusion $\mathfrak g \hookrightarrow\mathfrak g \rtimes G$ corresponds to $\mathfrak g = T_eG \hookrightarrow TG$, the inclusion $G\hookrightarrow \mathfrak g \rtimes G$ corresponds to $0_{TG}:G\to TG$, and the projection $\mathfrak g \rtimes G \to G$ corresponds to $\pi_{TG}:TG \to G$.

\subsection{Grothendieck groups}
\label{subsec: Grothendieck}

Let $\mathscr C$ be an essentially small category (one where the collection of isomorphism classes of objects in $\mathscr C$ is a set) with finite products. The set $\mathscr C/\cong$ of isomorphism classes of objects in $\mathscr C$ forms a commutative monoid under the operation $\left[E\right] + \left[E'\right] \equiv \left[E\times E'\right]$. We denote by $K\left(\mathscr C\right)$ the associated abelian group defined by the Grothendieck construction. If $\mathscr C$ is an additive category then this agrees with the standard notion of the `split' Grothendieck group of $\mathscr C$, i.e. the abelian group generated by isomorphism classes of objects and relations $\left[A\right]+\left[B\right]=\left[A\oplus B\right]$.

If $\mathcal F:\mathscr C\to \mathscr C'$ is a product preserving functor between categories satisfying the above assumptions then there is a group homomorphism $K\left(\mathcal F\right):K\left(\mathscr C\right) \to K\left(\mathscr C'\right)$
defined by $\left[E\right] - \left[E'\right] \mapsto \left[\mathcal F\left(E\right)\right] - \left[\mathcal F\left(E'\right)\right]$. The group homomorphism $K\left(\mathcal F\right)$ depends functorially on $\mathcal F$. If $\mathcal F,\mathcal F':\mathscr C \to \mathscr C'$ are naturally isomorphic functors then $K\left(\mathcal F\right) = K\left(\mathcal F'\right)$.

We write $KO_G\left(X\right)$ for $K\left(\mathrm{Vect}_G\left(X\right)\right)$ and  $K_G\left(X\right)$ for $K\left(\mathrm{Vect}_G^\mathbb C\left(X\right)\right)$. This agrees with $G$-equivariant topological $K$-theory as defined by Segal in \cite{segal} if both $G$ and $X$ are compact.


\section{Affine actions}
\label{sec: affine}

Throughout the paper, $G$ denotes a Lie group and $X$ denotes a $G$-manifold with action $t:G\times X \to X$. 

\subsection{Affine actions}
\label{subsection: affine actions}

\begin{definition}
\label{def: affine actions}
An \emph{affine action} of $TG$ on a real vector bundle $E \to X$ is a left action $\mu:TG \times E \to E$ of the Lie group $TG$ on the total space of $E$ such that $\mu$ is a vector bundle morphism covering $t:G\times X \to X$.

\end{definition}

\begin{example}
The derivative of $t$ defines an affine action $t_\ast : TG \times TX \to TX$ of $TG$ on $TX$.
\end{example}

\begin{remark}
\label{rem: affine}
The condition that $\mu$ is a morphism of vector bundles covering $t$ is equivalent to requiring that:
\begin{enumerate} 
\item The following diagram commutes:
\begin{equation}
\label{diag: affine action}
\xymatrix
{
TG \times E   \ar[d]  \ar[r]^-{\mu}  & E \ar[d] \\
G \times X \ar[r]_-{t}  & X
}
\end{equation}
\item For all $\left(g,x\right) \in G\times X$ the restriction of $\mu$ to the fibre $T_gG \oplus E_x$ over $\left(g,x\right)$ is a linear map $T_g G \oplus E_x \to E_{g \cdot x}$.
\end{enumerate}
The second of these conditions implies that $v_g \in T_g G$ acts on $e_x \in E_x$ by the composite map
$e_x  \mapsto \left(v_g,e_x\right)  \mapsto \mu\left(v_g,e_x\right)$. This is an affine linear map from $E_x$ to $E_{g\cdot x}$, which justifies the name \emph{affine action}.
\end{remark}

\begin{definition}
\label{def: morphisms of affine actions}
A \emph{morphism} from an affine action $\mu : TG \times E \to E$ to an affine action $\mu' : TG \times E'\to E'$ is a morphism $\psi : E \to E'$ of vector bundles over $X$ which is $TG$ equivariant, i.e. the following diagrams commute:
\[
\xymatrix{
E \ar[dr]_{\pi_E} \ar[r]^{\psi} & E' \ar[d]^{\pi_{E'}}  && TG \times E \ar[d]_{\mathrm{id} \times \psi} \ar[r]^-{\mu} & E \ar[d]^{\psi} \\
& X  &&TG \times E'  \ar[r]_-{\mu'}  & E'
}
\]
\end{definition}

\begin{definition}
Affine actions of $TG$ over $X$ form a category $\mathrm{Vect}_{TG}^{\mathrm{aff}}\left(X\right)$.
\end{definition}

\begin{remark}
If one considers $TG$ as a group object in the category of vector bundles, then affine actions coincide with the notion of actions of group objects - see, for example \cite{maclane}.
\end{remark}

\subsection{The structure of affine actions}
\label{subsec: decomp}

The following Lemma \ref{lemma: decomp} and Propositions \ref{prop: decomp}, \ref{prop: decomp2} and \ref{prop: decomp homs} describe how an affine action $TG\times E \to E$ can be decomposed into a $G$-action $G\times E \to E$ and a linear map $\mathfrak g \to \Gamma\left(E\right)$, and how morphisms between affine actions can be described in terms of this decomposition. These Propositions will be used in section \ref{subsec: the category} to prove our first main result, Theorem \ref{thm: A}. The first proposition shows that an affine action can be recovered from its restriction to $G$ and $\mathfrak g$, as motivated by Example \ref{ex: intro tangent} in the introduction.

\begin{lemma}
\label{lemma: decomp}
If $\mu:TG\times E \to E$ is an affine action on a vector bundle $E\to X$ then
\begin{equation*}
\mu\left(v_g,e_x\right) = \mu\left(0_g,e_x\right) + \mu\left( \left(R_{g^{-1}}\right)_\ast v_g,0_{g\cdot x}\right)
\end{equation*}
\end{lemma}

\begin{proof}
Using the fact that $\mu$ is fibrewise linear we have
\begin{align*}
\mu\left(v_g,e_x\right) & =
\mu\left(v_g + 0_g,e_x+0_x\right) \\
& = \mu\left(   \left(0_g,e_x\right) + \left(v_g,0_x\right) \right) \\
& = \mu\left(0_g,e_x\right) + \mu\left(v_g,0_x\right)
\end{align*}
Using the fact that $\mu$ is a left action of $TG$ we have
\begin{align*}
\mu\left(v_g,0_x\right) & = \mu\left(v_g \bullet 0_{g^{-1}} \bullet 0_g , 0_x\right) \\
& = \mu \left(   \left(R_{g^{-1}}\right)_\ast v_g, \mu\left(0_g,0_x\right)\right) \\
& = \mu \left(   \left(R_{g^{-1}}\right)_\ast v_g, 0_{g\cdot x}\right)
\end{align*}
\end{proof}

\begin{proposition}
\label{prop: decomp}
Let $E\to X$ be a vector bundle. There is a bijection 
\begin{align}
\nonumber
& \left\{ \text{affine actions } \mu:TG \times E \to E \right\} \\
\label{eqn: bijection}
 \xrightarrow{\cong} &
\left\{ \text{pairs } \left(\mu_G,\rho_\mathfrak g\right)
\text{satisfying  conditions } \left(\star\right),\left(\star\star\right) \text{ below } \right\}
\end{align}
where:
\begin{itemize}
\item[$\left(\star\right)$] $\mu_G : G \times E \to E$ is a left action of $G$ on $E$ making $E$ into a $G$-equivariant vector bundle over $X$, 
\item[$\left(\star\star\right)$] $\rho_\mathfrak g : \mathfrak g \to \Gamma\left(E\right)$ is a $G$-equivariant linear map, where  the $G$-action on $\Gamma\left(E\right)$ is induced from $\mu_G$.
\end{itemize}
The bijection \eqref{eqn: bijection} maps an affine action $\mu$ to the pair $\left(\mu_G,\rho_\mathfrak g\right)$ defined by
\begin{align}
\label{eqn: decomp1}
\mu_G \left(g,e_x\right) & = \mu\left(0_g,e_x\right) \\
\label{eqn: decomp2}
\rho_\mathfrak g \left(\alpha\right) & = \left(x \mapsto \mu\left(\alpha,0_x\right) \right)
\end{align}
The inverse of \eqref{eqn: bijection} maps a pair $\left(\mu_G,\rho_\mathfrak g\right)$ satisfying $\left(\star\right),\left(\star\star\right)$ to the affine action $\mu$ defined by
\begin{equation}
\label{eqn: reconstruct}
\mu\left(v_g,e_x\right) = \mu_G\left(g,e_x\right) + \rho_\mathfrak g\left( \left(R_{g^{-1}}\right)_\ast v_g\right)_{g \cdot x}
\end{equation}
for $v_g \in T_gG$ and $e_x \in E_x$.
\end{proposition}

Note that the addition on the right hand side of \eqref{eqn: reconstruct} is defined as the assumptions on $\mu_G$ and $\rho_\mathfrak g$ imply that both terms are elements of $E_{g\cdot x}$.

\begin{proof}
We first show that if $\mu:TG\times E \to E$ is an affine action then the pair $\left(\mu_G,\rho_\mathfrak g\right)$ defined by \eqref{eqn: decomp1} and \eqref{eqn: decomp2} satisfies $\left(\star\right)$ and $\left(\star\star\right)$, which shows that the map \eqref{eqn: bijection} is well defined. That $\mu_G$ is a left action of $G$ follows from the facts that $\mu$ is equal to the composite map
\begin{equation*}
G \times E \xrightarrow{0_{TG}\times \mathrm{id}_E} TG \times E \xrightarrow{\mu} E
\end{equation*}
and that $0_{TG}:G \to TG$ is a Lie group homomorphism. The fact that $\mu$ is a vector bundle morphism covering $t:G\times X \to X$ implies that for fixed $g\in G$ the map $\mu_G\left(g,-\right) : E \to E$ is a vector bundle morphism covering $t\left(g,-\right):X\to X$. 
This shows that $\mu_G$ satisfies $\left(\star\right)$. We will sometimes denote this $G$ action by $g \cdot e_x = \mu_G\left(g,e_x\right)$.

If $\alpha \in \mathfrak g=T_eG$ then the commutativity of \eqref{diag: affine action} and the fact that $\left(\alpha,0_x\right) \in T_eG \oplus E_x$ implies that $\mu\left(\alpha,0_x\right) \in E_x$, so that the map $\rho_\mathfrak g\left(\alpha\right) = \left(x \mapsto \mu\left(\alpha,0_x\right)\right)$ is a smooth section of $E$. The fibrewise linearity of $\mu$ implies that the map $\alpha \mapsto \rho_\mathfrak g \left(\alpha\right) \in \Gamma\left(E\right)$ is linear. If $g \in G$, $\alpha \in \mathfrak g$ and $x \in X$ then 
\begin{align*}
\left(\rho_\mathfrak g \left(\mathrm{Ad}_g \alpha \right)\right)_{x}
& =
\mu\left(0_g\bullet \alpha \bullet 0_{g^{-1}},0_x\right) \\
& =
0_g \cdot \mu\left(\alpha,0_{g^{-1}} \cdot 0_x\right) \\
& =
0_g \cdot \mu\left(\alpha,0_{g^{-1}\cdot x}\right) \\
& = 
g \cdot \rho_\mathfrak g \left(\alpha\right)_{g^{-1}\cdot x} \\
& =
\left(g \cdot \rho_\mathfrak g\left(\alpha\right)\right)_x
\end{align*}
This shows that $\rho_\mathfrak g$ satisfies $\left(\star\star\right)$.

Now let $\left(\mu_G,\rho_\mathfrak g\right)$ be a pair satisfying $\left(\star\right)$ and $\left(\star\star\right)$, and let us show that $\mu$ as defined in \eqref{eqn: reconstruct} defines a fibrewise linear left action of $TG$. That $\mu$ is fibrewise linear follows from the fibrewise linearity of $\mu_G$ and the linearity of $\rho_\mathfrak g$. Transporting the action from $TG$ to $\mathfrak g\rtimes G$ via the Lie group isomorphism $v_g \mapsto \left(\left(R_{g^{-1}}\right)_\ast v_g,g\right)$, we have
 \begin{align*}
\left(\alpha,g\right) \cdot e_x 
& =  g \cdot e_x + \rho_\mathfrak g \left(\alpha\right)_{g\cdot x} 
\end{align*}
If $\left(\alpha,g\right),\left(\beta,h\right) \in G \ltimes \mathfrak g$ then using $\left(\star\right)$ and $\left(\star\star\right)$ we have
\begin{align*}
\left(\alpha,g\right) \cdot \left( \left(\beta,h\right) \cdot e_x \right) 
& = 
 \left(\alpha,g\right)\cdot \left( h \cdot e_x + \rho_\mathfrak g \left(\beta\right)_{h \cdot x}    \right) \\
& = 
 g \cdot \left(h \cdot e_x\right) +
g \cdot \rho_\mathfrak g \left(\beta\right)_{h \cdot x} + \rho_\mathfrak g \left(\alpha\right)_{gh\cdot x} \\
& =   gh \cdot e_x +
\rho_\mathfrak g \left(\mathrm{Ad}_g\beta\right)_{gh \cdot x} + \rho_\mathfrak g \left(\alpha\right)_{gh\cdot x} \\
& =
gh \cdot \left( e_x\right) +
\rho_\mathfrak g \left(  \alpha+ \mathrm{Ad}_g \beta \right)_{gh\cdot x}    \\
& = 
\left(\alpha+\mathrm{Ad}_g\beta,gh\right) \cdot e_x \\
& = \left(\left(\alpha,g\right) \bullet \left(\beta,h\right) \right) \cdot e_x
\end{align*} 
This shows that $\mu$ defines a left action of $TG$.

It remains to show that \eqref{eqn: bijection} is a bijection with inverse defined by \eqref{eqn: reconstruct}. It follows from Lemma \ref{lemma: decomp} that for an affine action $\mu$  mapped to $\left(\mu_G,\rho_\mathfrak g\right)$ by \eqref{eqn: bijection} we have
\begin{align*}
\mu_G \left(g,e_x\right) + \rho_\mathfrak g\left( \left(R_{g^{-1}}\right)_\ast v_g\right)_{g\cdot x}
& = 
\mu\left(0_g,e_x\right) + \mu\left( \left(R_{g^{-1}}\right)_\ast v_g,0_{g\cdot x}\right) \\
& = \mu\left(v_g,e_x\right) 
\end{align*}
Conversely, if $\left(\mu_G,\rho_\mathfrak g\right)$ is a pair satisfying $\left(\star\right),\left(\star\star\right)$, and $\mu$ is the affine action defined by \eqref{eqn: reconstruct}, then
\begin{align*}
\mu\left(0_g,e_x\right) 
& = \mu_G\left(g,e_x\right) + \rho_\mathfrak g\left( \left(R_{g^{-1}}\right)_\ast 0_g\right)_{g \cdot x} \\
& = \mu_G\left(g,e_x\right) + 0_{g\cdot x} \\
& = \mu_G\left(g,e_x\right)  
\end{align*}
and
\begin{align*}
\mu\left(\alpha,0_x\right) 
& =  \mu_G\left(e,0_x\right) + \rho_\mathfrak g\left( \left(R_{e}\right)_\ast \alpha\right)_{e \cdot x} \\
& = 0_x + \rho_\mathfrak g\left(\alpha\right)_x \\
& = \rho_\mathfrak g \left(\alpha\right)_x
\end{align*}
\end{proof}


\begin{remark}
It follows from Lemma \ref{lemma: decomp} and the proof of Proposition \ref{prop: decomp} that if one transports an affine action $\mu:TG\times E \to E$ from $TG$ to $\mathfrak g \rtimes G$ via the Lie group isomorphism $v_g \mapsto \left(g,\left(R_{g^{-1}}\right)_\ast v_g\right)$ then
 \begin{align}
\label{eqn: action of Gg2}
\left(\alpha,g\right) \cdot e_x 
& =  g \cdot e_x + \rho_\mathfrak g \left(\alpha\right)_{g\cdot x}  \\
& = \mu\left(0_g,e_x\right) + \mu \left(\alpha,0_{g\cdot x}\right)
\end{align} 
In particular; $g \cdot e_x = \mu\left(0_g,e_x\right)$ and $
\alpha \cdot e_x  = e_x + \mu\left(\alpha,0_x\right)$, so that elements of $\mathfrak g$ act as fibrewise affine linear transformations. This motivates the name \emph{affine action}.
\end{remark}

\begin{proposition}
\label{prop: decomp2}
Let $E\to X$ be a $G$-equivariant vector bundle with $G$-action $\mu_G:G\times E \to E$. There is a bijection 
\begin{align}
\nonumber
& \left\{\rho_\mathfrak g: \mathfrak g \to \Gamma\left(E\right) \; | \;
\rho_\mathfrak g \text{ is a } G \text{-equivariant linear map} \right\} \\
\label{eqn: bijection2}
\xrightarrow{\cong} & 
\left\{\phi:\mathfrak g_X  \to E \; | \;  \phi \text{ is a morphim of } G \text{-equivariant vector bundles}\right\}
\end{align}

The bijection \eqref{eqn: bijection2} maps $\rho_\mathfrak g : \mathfrak g \to \Gamma\left(E\right)$ to the morphism $\phi:\mathfrak g_X \to E$ defined by
\begin{align*}
\phi\left(x,\alpha\right) & = \rho_\mathfrak g\left(\alpha\right)_x
\end{align*}
\end{proposition}

\begin{proof}
This follows from the bijection \eqref {eqn: hom(VX,E)} applied to the $G$-representation $\mathfrak g$ and the $G$-equivariant vector bundle $E$.
\end{proof}

\begin{proposition}
\label{prop: decomp homs}
Let $\mu:TG\times E\to E$ and $\mu':TG\times E' \to E'$ be affine actions corresponding to pairs $\left(\mu_G,\rho_\mathfrak g\right),\left(\mu_G,\phi\right)$ and $\left(\mu'_G,\rho'_\mathfrak g\right),\left(\mu'_G,\phi'\right)$ under the bijections of Propositions \ref{prop: decomp} and \ref{prop: decomp2}. Let $\psi : E\to E'$ be a morphism of vector bundles over $X$. The following are equivalent:
\begin{enumerate}
\item $\psi$ is $TG$-equivariant.
\item $\psi$ is $G$-equivariant and $\Gamma\left(\psi\right) \circ \rho_{\mathfrak g} = \rho'_\mathfrak g$.
\item $\psi$ is $G$-equivariant and $\psi \circ \phi = \phi'$.
\end{enumerate}
\end{proposition}

\begin{proof}
$\left(1\Leftrightarrow 2\right)$.
Via the isomorphism $TG \cong  \mathfrak g\rtimes G$, a map $\psi:E\to E'$ is $TG$-equivariant if and only if it is $\mathfrak g\rtimes G$-equivariant for the action \eqref{eqn: action of Gg2}. If $\left(\alpha,g\right) \in \mathfrak g\rtimes G$ then
\begin{align*}
\psi \left( \left(\alpha,g\right) \cdot e_x \right)  
& =
\psi \left( g \cdot e_x + \rho_\mathfrak g(\alpha)_{g\cdot x} \right) \\
& =
 \psi\left(g \cdot e_x\right) + \psi \left(  \rho_\mathfrak g(\alpha)_{g\cdot x}\right)  
\end{align*}
and
\begin{align*}
\left(\alpha,g\right) \cdot \psi \left(e_x\right)  
& =  g \cdot \psi\left(e_x\right) + \rho_\mathfrak g'(\alpha)_{g \cdot x} 
\end{align*}
Therefore, $\psi$ is $TG$-equivariant if and only if
\begin{align}
\label{eqn: morphism condition}
\psi\left(g \cdot e_x\right) + \psi \left(  \rho_\mathfrak g(\alpha)_{g\cdot x}\right) 
& = 
g \cdot \psi\left(e_x\right) + \rho_\mathfrak g'(\alpha)_{g \cdot x} 
\end{align}
Setting first $\alpha = 0$, and then $g=e$ and $e_x = 0$, one sees that \eqref{eqn: morphism condition} is equivalent to the two equations
\begin{align}
\label{eqn: morphism condition 1}
\psi \left( g \cdot e_x\right) & = g \cdot \psi\left(e_x\right) \\
\label{eqn: morphism condition 2}
\psi\left( \rho_\mathfrak g (\alpha) \right) & = \rho_\mathfrak g' (\alpha)
\end{align}
Equation \eqref{eqn: morphism condition 1} is the condition that $\psi$ is $G$-equivariant, and \eqref{eqn: morphism condition 2} is the condition that $\Gamma\left(\psi\right) \circ \rho_\mathfrak g = \rho_\mathfrak g'$.

$\left(2\Leftrightarrow 3\right)$. If $\alpha \in \mathfrak g$ and $x \in X$ then
\[
\left(\Gamma\left(\psi\right) \circ \rho_\mathfrak g \left(\alpha\right) \right)_x = \psi\left(\rho_\mathfrak g \left(\alpha\right)_x\right) = \psi \left(\phi\left(x,\alpha\right)\right)
\]
and
\[
\rho'_\mathfrak g\left(\alpha\right)_x = \phi'\left(x,\alpha\right)
\]
Therefore, $\Gamma\left(\psi\right) \circ \rho_\mathfrak g = \rho'_\mathfrak g$ if and only if $\psi \circ \phi = \phi'$.
\end{proof}

\subsection{The category of affine actions}
\label{subsec: the category}

Recall that for a fixed object $m$ in a category $\mathscr C$, the \emph{over-slice category} $m\backslash \mathscr C$ is defined as follows: 
\begin{itemize}
\item the objects are pairs $\left(a,\phi\right)$, where $a$ is an object in $\mathscr C$ and $\phi:m\to a$ is a morphism in $\mathscr C$,
\item the morphisms $\left(a,f\right) \to \left(a',f'\right)$ are morphisms $\chi:a \to a'$ in $\mathscr C$ such that $\chi\circ f = f'$:
\[
\xymatrix{
m \ar[r]^{f}  \ar[dr]_{f'}  &   a \ar[d]^{\chi} \\
&  a'
}
\]
\item the composition of morphisms is induced from that of $\mathscr C$.
\end{itemize}
There is a canonical faithful forgetful functor $m\backslash \mathscr C \to \mathscr C$ defined by $\left(a,f\right) \mapsto a$ and $\left(\chi:\left(a,f\right) \to \left(a',f'\right)\right) \mapsto \left(\chi:a \to a'\right)$.

\begingroup
\setcounter{mainthm}{0} 
\renewcommand\themainthm{\Alph{mainthm}}
\begin{mainthm}
\label{thm: A}
The following three categories are isomorphic:
\begin{enumerate}
\item The category $\mathrm{Vect}_{TG}^\mathrm{aff}\left(X\right)$ of affine actions of $TG$ over $X$.
\item The category of pairs $\left(E,\rho_\mathfrak g\right)$, defined as follows: 
\begin{itemize}
\item The objects are pairs $\left(E,\rho_\mathfrak g\right)$, where $E$ is a $G$-equivariant vector bundle over $X$ and $\rho_\mathfrak g : \mathfrak g \to \Gamma\left(E\right)$ is a $G$-equivariant linear map. 
\item The morphisms $\left(E,\rho_\mathfrak g\right) \to \left(E',\rho_\mathfrak g\right)$ are morphisms $\psi:E\to E'$ of $G$-equivariant vector bundles over $X$ such that $\Gamma\left(\psi\right)\circ \rho_\mathfrak g = \rho'_\mathfrak g$.
\item Composition is given by composition of morphisms of vector bundles over $X$.
\end{itemize}
\item The slice category $ \mathfrak g_X \backslash \mathrm{Vect}_G\left(X\right)$.
\end{enumerate}
\end{mainthm}
\endgroup

\begin{proof}
The bijections of Propositions \ref{prop: decomp} and \ref{prop: decomp2} provide bijections between the classes of objects of each of these three categories. The bijections of Proposition \ref{prop: decomp homs} provide bijections between hom-sets. These bijections are functorial because the composition of morphisms in all three categories is given by the composition of morphisms of vector bundles over $X$.
\end{proof}

\begin{remark}
For the remainder of the paper we shall use the isomorphisms of Theorem \ref{thm: A} implicitly. We shall reserve the notation $\left(E,\rho_\mathfrak g\right)$ and $\left(E,\phi\right)$ for objects in the second and third categories in Theorem \ref{thm: A} respectively, and refer to either as a \emph{pair}. In particular, Theorem \ref{thm: A} shows that associated to every affine action $TG\times E \to E$ is an `underlying'  $G$-vector bundle $E$, equal to the $E$ in either of the associated pairs. Explicitly, the underlying $G$-action $\mu_G$ is defined by the formula \eqref{eqn: decomp1} and is equal to the restriction of $\mu$ along the Lie group homomorphism $0_{TG}:G \to TG$.

In section \ref{sec: examples} we give a number of examples, some phrased in terms of affine actions and some in terms of pairs. In sections \ref{sec: functors} and \ref{sec: K} we work mostly with the category $\mathfrak g_X \backslash \mathrm{Vect}_G\left(X\right)$, but where it is easy to do so describe the corresponding statements in terms of affine actions - see Remarks \ref{rem: UforTG}, \ref{rem: analogue thmB} and \ref{rem: TGpullback}, and Corollaries \ref{cor: TGterminal} and \ref{cor: TGproduct}.
\end{remark}

\section{Examples and special cases}
\label{sec: examples}

\subsection{Examples}
\label{subsection: examples}

\begin{example}
\textbf{The tangent bundle.}
The composite map
\[
TG \times TX \xrightarrow{\cong} T\left(G\times X\right) \xrightarrow{t_\ast} TX
\]
is an affine action on $TX$. The corresponding pair $\left(\mu_G,i_\mathfrak g\right)$ is given by $g \cdot w  = \left(t_g\right)_\ast \left(w\right)$ and $i_\mathfrak g \left(\alpha\right) = \alpha^\#$, where $\alpha^\sharp \in \Gamma\left(TX\right)$ is the induced vector field associated to $\alpha$, defined by 
\[
\alpha^\sharp_x = t_\ast\left(\alpha,0_x\right)=\tfrac{d}{dt}|_{t=0}\left(\mathrm{exp}\left(t\alpha\right)\cdot x\right)
\] 
In terms of $\mathfrak g \rtimes G$ we have
\[
\left(\alpha,g\right) \cdot w_x = \left(t_g\right)_\ast w_x + \alpha^\#_{g\cdot x}
\]
\end{example}

\begin{example}
\label{example: equivariant bundles}
\textbf{Equivariant vector bundles.} 
If $E$ is a $G$-equivariant vector bundle with $G$-action $\mu_G : G \times E \to E$ then the composite map
\begin{align*}
TG \times E \xrightarrow{\pi_G \times \mathrm{id}} G \times E \xrightarrow{\mu_G} E
\end{align*}
is an affine action. The corresponding pair is $\left(\mu_G,0\right)$. In terms of $\mathfrak g \rtimes G$ we have
\[
\left(\alpha,g\right) \cdot e_x = g \cdot e_x
\]
\end{example}

\begin{example}
\label{example: action algebroid}
\textbf{The $G$-bundle $\mathfrak g_X$.}
The $G$-vector bundle $\mathfrak g_X$ defines a canonical object $\left(\mathfrak g_X,\mathrm{id}\right)$
in the category $\mathfrak g_X\backslash \mathrm{Vect}_G\left(X\right)$. The corresponding affine action is given by
\[
v_g\cdot \left(x,\beta\right) 
=
\left(g\cdot x,\mathrm{Ad}_g \beta + \left(R_{g^{-1}}\right)_\ast v_g\right)
\]
In terms of $\mathfrak g \rtimes G$ we have
\[
\left(\alpha,g\right) \cdot \left(x,\beta\right) 
=
\left(g\cdot x,\mathrm{Ad}_g \beta + \alpha \right)
\]
\end{example}

\begin{example}
\textbf{$G$-modules.}
As a generalization of Example \ref{example: action algebroid}, suppose that $M$ is a finite dimensional real $G$-representation and $\bar\phi:\mathfrak g\to M$ is a $G$-map. Then $\bar\phi$ extends to a constant morphism of $G$-vector bundles $\phi: \mathfrak g_X \to M_X$ and $\left(M_X,\phi\right)$ is an object in $\mathfrak g_X \backslash \mathrm{Vect}_G\left(X\right)$. The corresponding affine action is given by
\[
v_g \cdot \left(x,m\right) = \left(g\cdot x, g\cdot m + \bar\phi\left( \left(R_{g^{-1}}\right)_\ast v_g \right) \right)
\]
In terms of $\mathfrak g \rtimes G$ we have
\[
\left(\alpha,g\right) \cdot \left(x,m\right) = \left(g\cdot x, g \cdot m + \bar\phi\left(\alpha\right) \right)
\]
This construction extends to a functor
\[
\mathfrak g \backslash \mathrm{Rep}\left(G\right)
\to
\mathfrak g_X \backslash \mathrm{Vect}_G\left(X\right) 
\cong
\mathrm{Vect}_{TG}^{\mathrm{aff}}\left(X\right)
\]
which coincides with the pullback functor $\mathrm{Vect}_{TG}^{\mathrm{aff}}\left(\ast\right) \to \mathrm{Vect}_{TG}^{\mathrm{aff}}\left(X\right)$ associated to the $G$-map $X\to \ast$ (see Definition \ref{def: pullback} in section \ref{subsec: pullback functors} below).
\end{example}

\begin{example}
\label{ex: Liealg}
\textbf{Equivariant Lie algebroids.}
Recall that a \emph{Lie algebroid} is a vector bundle $A\to X$ equipped with an $\mathbb R$-linear Lie bracket on $\Gamma\left(A\right)$ and a vector bundle morphism $\rho: A\to TX$, such that the Leibniz rule $\left[\xi,f\xi'\right]=\rho\left(\xi\right) \left(f\right)\xi' + f\left[\xi,\xi'\right]$ is satisfied for all $\xi,\xi'\in \Gamma\left(A\right)$ and $f\in C^\infty\left(X\right)$. See \cite{mackenzie} for more details. 

A  $G$-\emph{equivariant Lie algebroid} (called a \emph{Harish-Chandra Lie algebroid} in \cite{bb}) is a Lie algebroid $A\to X$ equipped with an action $G\times A \to A$ of $G$ on $A$ by Lie algebroid automorphisms and a $G$-equivariant Lie algebra morphism $i_\mathfrak g: \mathfrak g \to \Gamma\left(A\right)$ such that $\alpha \cdot \xi = \left[i_\mathfrak g\left(\alpha\right), \xi\right]$ for all $\alpha \in \mathfrak g$ and $\xi \in \Gamma\left(A\right)$, where $\xi \mapsto \alpha \cdot \xi$ is the action of $\mathfrak g$ on $\Gamma\left(A\right)$ given by the derivative of the action of $G$. Variants of this notion have appeared in \cite{alekseev},\cite{bruzzo},\cite{marrero},\cite{ginzburg}.

By forgetting the Lie brackets, every $G$-equivariant Lie algebroid gives rise to a pair $\left(A,i_\mathfrak g\right)$ and therefore to an affine action of $TG$. Affine actions of this type were considered in \cite{marrero}.
\end{example}

\begin{example}
\textbf{Affine actions on $\mathbb R_X$}
Let $\mathbb R_X$ be the $G$-vector bundle associated to the trivial representation of $G$. There is a bijection between affine actions with underlying $G$-vector bundle $\mathbb R_X$, and the space $C^\infty\left(X,\mathfrak g^\ast\right)^G$. This follows from the bijection \eqref{eqn: hom(VX,E)} and the isomorphism $\mathrm{Hom}\left(\mathfrak g,\mathbb R\right) \cong \mathfrak g^\ast$ of $G$-representations:
\[
\mathrm{Hom}_{\mathrm{Vect}_G\left(X\right)}\left(\mathfrak g_X,\mathbb R_X\right)
\cong
C^\infty\left(X, \mathrm{Hom}\left(\mathfrak g,\mathbb R\right) \right)^G
\cong
C^\infty\left(X,\mathfrak g^\ast\right)^G
\]
Given a function $f \in C^\infty\left(X,\mathfrak g^\ast\right)^G$ the corresponding affine action is given by
\[
v_g \cdot \left(x,\lambda\right) = \left(g\cdot x, \lambda + f_x\left( \left(R_{g^{-1}}\right)_\ast v_g \right) \right)
\]
In terms of $\mathfrak g \rtimes G$ we have
\[
\left(\alpha,g\right) \cdot \left(x,\lambda\right) = \left(g\cdot x, \lambda + f_x\left(\alpha\right) \right)
\]
\end{example}

\begin{example}
\label{ex: products}
\textbf{Products.} If $\left(E,\phi\right)$ and $\left(E',\phi'\right)$ are objects in $\mathfrak g_X\backslash \mathrm{Vect}_G\left(X\right)$ then so is $\left(E\oplus E',\left(\phi,\phi'\right)\right)$. In fact, 
\[
\left(E\oplus E',\left(\phi,\phi'\right)\right) = \left(E,\phi\right) \times \left(E',\phi'\right),
\] 
where the right hand side is the product of $\left(E,\phi\right)$ and $\left(E',\phi'\right)$ in the category $\mathfrak g_X\backslash\mathrm{Vect}_G\left(X\right)$ (see section \ref{subsec: products}). The corresponding affine action is the diagonal action 
\[
v_g\cdot \left(e_x,e'_x\right) = \left(v_g\cdot e_x,v_g\cdot e'_x\right)
\]
In terms of $\mathfrak g \rtimes G$ we have
\[
v_g\cdot \left(e_x,e'_x\right) = \left(g\cdot e_x+\rho_\mathfrak g\left(\alpha\right)_x,g\cdot e'_x + \rho'_\mathfrak g\left(\alpha\right)_x\right)
\]

\end{example}

\subsection{Special cases}

\begin{example}
\textbf{Discrete groups.}
If $G$ is a discrete group then $TG \cong G$, and affine actions are the same as equivariant vector bundles.
\end{example}

\begin{example}
\textbf{Tori.}
Suppose that $G$ is one dimensional and abelian. In this case the adjoint representation is one dimensional and trivial.
This implies that for every $G$-equivariant vector bundle $E$ there is an isomorphism
$\mathrm{Hom}_G\left(\mathfrak g,\Gamma\left(E\right)\right) 
 \cong \Gamma\left(E\right)^G$. Affine actions are therefore equivalent to pairs $\left(E,\xi\right)$, where $E$ is a $G$-equivariant vector bundle and $\xi$ is a $G$-invariant section of $E$. 
\end{example}

\begin{example}
\label{example: X=point}
\textbf{Points.}
If $X=\ast$ is a point then $\mathfrak g_X \backslash \mathrm{Vect}_G\left(X\right) = \mathfrak g\backslash \mathrm{Rep}\left(G\right)$.
\end{example}

Assuming that $G$ is simple and compact, one can use the description of products of Example \ref{ex: products} to describe the objects in $\mathfrak g \backslash \mathrm{Rep}\left(G\right)$ explicitly:

\begin{proposition}
\label{prop: K(point)}
If $G$ is a simple compact Lie group and $\left(V,\phi\right)$ is an object in $\mathfrak g \backslash \mathrm{Rep}\left(G\right)$ then either:
\begin{enumerate}
\item 
$\left(V,\phi\right) = \left(V,0\right)$, or
\item 
$\left(V,\phi\right) \cong \left(\mathfrak g,\mathrm{id}\right)^n \times \left(W,0\right)$
\end{enumerate} 
where $n \in \mathbb N$ and $W$ is a finite dimensional $G$-representation with no summand isomorphic to $\mathfrak g$. 
\end{proposition}

\begin{remark}
\label{rem: K(point)}
Note that in the setting of Proposition \ref{prop: K(point)}, if  $\left(V,\phi\right) = \left(V,0\right)$ then $\left(V,\phi\right) \cong \left(\mathfrak g,0\right)^n \times \left(W,0\right)$ for some $n\in \mathbb Z^{\geq 0}$ and $W$ a finite dimensional $G$-representation with no summand isomorphic to $\mathfrak g$. 
\end{remark}

\begin{proof}
The proof is an application of complete reducibility ($G$ is compact) and Schur's Lemma ($\mathfrak g$ is irreducible as a $G$-representation because $G$ is simple). Note that as $G$ is simple and compact the adjoint representation $\mathfrak g$ of $G$ is absolutely irreducible (\cite{onishchik}). In particular, $\mathrm{End}_G\left(\mathfrak g\right)  \cong \mathbb R$.

Let $\left(V,\phi\right)$ be an object in $\mathfrak g \backslash \mathrm{Rep}\left(G\right)$. As a $G$-representation $V \cong \mathfrak g^{\oplus n} \oplus W$ for some non-negative integer $n$ and $G$-representation $W$ with no summand isomorphic to $\mathfrak g$. It follows from Schur's Lemma that there is an isomorphism 
\begin{align}
\label{eqn: Schur}
\mathbb R^n & \xrightarrow{\cong}
\mathrm{Hom}_G\left(\mathfrak g, \mathfrak g^{\oplus n} \oplus W\right) \\
\nonumber
\underline \lambda & \mapsto \left(\alpha \mapsto  \left(\lambda_1 \alpha,\dots,\lambda_n\alpha,0\right)\right)
\end{align}
Suppose that $\left(V',\phi'\right) \cong \mathfrak g^{\oplus n'} \oplus W'$ is a second object in $\mathfrak g \backslash \mathrm{Rep}\left(G\right)$ and that under the bijection \eqref{eqn: Schur} $\left(V,\phi\right)$ (respectively $\left(V',\phi'\right)$) corresponds to $\left(\mathfrak g^{\oplus n} \oplus W, \underline \lambda\right)$ (respectively $\left(\mathfrak g^{\oplus n'} \oplus W', \underline \lambda'\right)$). If $\left(V,\phi\right) \cong \left(V',\phi'\right)$ then we necessarily have $n'=n$ and $W'\cong W$. Using Schur's Lemma again, we have a bijection
\begin{align*}
\mathrm{Aut}_G\left(\mathfrak g^{\oplus n}\oplus W\right) & \cong GL_n\left(\mathbb R\right) \times \mathrm{Aut}_G\left(W\right)
\end{align*}
from which it follows that isomorphisms in $\mathfrak g\backslash \mathrm{Rep}\left(G\right)$ from $\left(V,\phi\right)$ to $\left(V',\phi'\right)$ correspond to diagrams
\[
\xymatrix{
\mathfrak g\ar[r]^-{\underline \lambda} 
\ar[dr]_{\underline \lambda'} 
& \mathfrak g^{\oplus n}\oplus W\ar[d]^{\left(A,\psi\right)} \\
& \mathfrak g^{\oplus n}\oplus W
}
\]
in which $\left(A,\psi\right) \in GL_n\left(\mathbb R\right) \times \mathrm{Aut}_G\left(W\right)$ and $A\underline \lambda = \underline \lambda'$. The Proposition now follows from the fact that there are exactly two $GL_n\left(\mathbb R\right)$ orbits in $\mathbb R^n$: the zero orbit $\left\{\underline 0 \right\}$, and its complement $\mathbb R^n \backslash \left\{\underline 0\right\}$.  Under the bijection \eqref{eqn: Schur} the zero vector $\underline 0$ corresponds to the zero morphism $\mathfrak g \to \mathfrak g^{\oplus n}\oplus W$ and therefore to the object $\left(\mathfrak g,0\right)^n \times \left(W,0\right) \cong \left(V,0\right)$, and the vector $\left(1,\cdots,1\right) \in \mathbb R^n \backslash \left\{\underline 0\right\}$ corresponds to the map $\left(\mathrm{diag},0\right):\mathfrak g \to \mathfrak g^{\oplus n}\oplus W$ and therefore to the object $\left(\mathfrak g,\mathrm{id}\right)^n \times \left(W,0\right)$.
\end{proof}

\begin{remark}
\label{rem: 0=1}
Surprisingly, in the Grothendieck group $K_{TG}^{\mathrm{aff}}\left(\ast\right)$ (see Definition \ref{def: KTG}) the objects $\left[\left(\mathfrak g,\mathrm{id}\right)\right]$ and  $\left[\left(\mathfrak g,0\right)\right]$ are in fact equal. See Remark \ref{rem: monoids vs groups} below for a further discussion of this point.
\end{remark}

\begin{example}
\textbf{Trivial $G$-spaces.} 
\label{example: X=trivial G-space}
As a generalization of Example \ref{example: X=point} consider the case where $X$ is a trivial $G$-space. In this case an affine action $TG\times E \to E$ over $X$ is equivalent to a smoothly varying family $TG \times E_x \to E_x$ of affine actions parametrised by $x\in X$. In particular, associated to each point $x\in X$ there is an object $\left(E_x,\phi_x\right)$ in $\mathfrak g\backslash \mathrm{Rep}\left(G\right)$.
\end{example}

\begin{proposition}
\label{example: homog}
\textbf{\emph{Homogeneous $G$-spaces.}}
Let $G/H$ be a homogeneous $G$-space. Then there is an equivalence of categories
\begin{equation}
\label{eqn: G/H}
\mathrm{Vect}_{TG}^\mathrm{aff}\left(G/H\right)
\simeq
\mathfrak g \backslash \mathrm{Rep}\left(H\right)
\end{equation}
where $\mathfrak g$ is considered as an $H$-module by restriction of the adjoint representation of $G$.
\end{proposition}

\begin{proof}
Under the equivalence $\mathrm{Vect}_G\left(G/H\right) \to \mathrm{Rep}\left(H\right)$ the $G$-vector bundle $\mathfrak g_X$ is mapped to the $H$-representation $\mathfrak g$. Therefore
\begin{equation*}
\mathrm{Vect}_{TG}^\mathrm{aff}\left(G/H\right) 
 \simeq 
\mathfrak g_X\backslash \mathrm{Vect}_G\left(X\right) 
 \simeq
\mathfrak g \backslash \mathrm{Rep}\left(H\right)
\end{equation*}
\end{proof}

\begin{example}\textbf{The tangent bundle of $G/H$.}
Under the equivalence \eqref{eqn: G/H} the tangent bundle $T\left(G/H\right)$ corresponds to the quotient map $\mathfrak g \to \mathfrak g/\mathfrak h$ in the category $\mathfrak g \backslash \mathrm{Rep}\left(H\right)$.
\end{example}

\begin{proposition}
\textbf{\emph{Free $G$-spaces.}}
Suppose that $P$ is a free, proper $G$-space so that $P/G$ is a manifold and $P \to P/G$ is a principal $G$-bundle. Then 
\[
\mathrm{Vect}_{TG}^{\mathrm{aff}} \left(P\right)
\simeq
\mathrm{ad}\left(P\right)\backslash \mathrm{Vect}\left(P/G\right)
\]
where $\mathrm{ad}\left(P\right) = P\times_G\mathfrak g$ is the adjoint bundle of $P$.

\end{proposition}

\begin{proof}
The quotient construction 
 $E \mapsto E/G$
yields an equivalence of categories $\mathrm{Vect}_G\left(P\right) \to \mathrm{Vect}\left(P/G\right)$ under which the $G$-vector bundle $\mathfrak g_P$ is mapped to the adjoint bundle $\mathrm{ad}\left(P\right)= P\times_G \mathfrak g$. Therefore
\begin{equation}
\label{eqn: quotient equiv}
\mathrm{Vect}_{TG}^{\mathrm{aff}} \left(P\right)
\simeq
\mathfrak g_P \backslash \mathrm{Vect}_G\left(P\right)
\simeq
\mathrm{ad}\left(P\right)\backslash \mathrm{Vect}\left(P/G\right)
\end{equation}
\end{proof}

\begin{example}
\textbf{Atiyah algebroids.} If $\pi:P \to X$ is a principal $G$-bundle then the \emph{Atiyah algebroid} of $P$ is the Lie algebroid (see Example \ref{ex: Liealg}) $TP/G$ over $X$ (see \cite{mackenzie}). The Atiyah algebroid fits into a short exact sequence
\begin{equation*}
0 \to \mathrm{ad}\left(P\right) \xrightarrow{} TP/G \xrightarrow{} TX \to 0 
\end{equation*}
which arises from an application of the quotient construction to the short exact sequence 
\begin{equation*}
0 \to \mathrm{Ker}\pi_\ast \xrightarrow{} TP \xrightarrow{\pi_\ast} \pi^\ast TX \to 0 
\end{equation*}
and the isomorphism of $G$-equivariant vector bundles $\mathrm{Ker}\pi_\ast \cong \mathfrak g_P$.
Under the equivalence \eqref{eqn: quotient equiv}, the tangent bundle $TP$ corresponds to the object $\left(TP/G, \mathrm{ad}\left(P\right)\to TP/G\right)$ in $\mathrm{ad}\left(P\right)\backslash \mathrm{Vect}\left(X\right)$. 
\end{example}



\section{$\mathrm{Vect}_{TG}^{\mathrm{aff}}\left(X\right)$ and $\mathrm{Vect}_G\left(X\right)$}
\label{sec: functors}

In this section we relate affine actions to equivariant vector bundles and in section \ref{subsec: thm B} prove our second main result Theorem \ref{thm: B}. We mostly describe the statements in terms of the category $\mathfrak g_X\backslash \mathrm{Vect}_G\left(X\right)$, but will indicate the corresponding results for $\mathrm{Vect}_{TG}^{\mathrm{aff}}\left(X\right)$ - see Remarks \ref{rem: UforTG} and \ref{rem: analogue thmB}.

\subsection{The functors $\mathrm U,\mathrm F$ and $\sigma$}
\label{subsec: forget}

We will define several functors between the categories $\mathfrak g_X \backslash \mathrm{Vect}_G\left(X\right)$ and $\mathrm{Vect}_G\left(X\right)$. In Remark \ref{rem: UforTG} below we explain what are the corresponding functors between $\mathrm{Vect}_{TG}^{\mathrm{aff}}\left(X\right)$ and $\mathrm{Vect}_G\left(X\right)$.

\begin{definition}
We denote by $\mathrm U$ the forgetful functor
\[
\mathrm U: \mathfrak g_X \backslash \mathrm{Vect}_G\left(X\right) \to \mathrm{Vect}_G\left(X\right)
\]
which maps a pair $\left(E,\phi\right)$ to the $G$-vector bundle $E$ and a morphism $\psi : \left(E,\phi\right) \to \left(E',\phi'\right)$ to the morphism $\psi:E\to E'$.
\end{definition}

\begin{definition}
\label{def: sigma}
We denote by $\sigma$ the functor
\[
\sigma: \mathrm{Vect}_G\left(X\right) \to 
\mathfrak g_X \backslash \mathrm{Vect}_G\left(X\right)
\]
defined on objects by $\sigma\left(E\right) = \left(E,0\right)$
and on morphisms by $\sigma \left( \psi : E \to E' \right) = \left(\psi: \left(E,0\right) \to \left(E',0\right)\right)$.
\end{definition}

\begin{definition}
\label{def: left adjoint}
We denote by $\mathrm F$ the functor
\[
\mathrm F : \mathrm{Vect}_G\left(X\right) \to \mathfrak g_X \backslash \mathrm{Vect}_G\left(X\right)
\]
which maps a $G$-vector bundle $E$ to the object $\left(\mathfrak g_X\oplus E,i_{\mathfrak g_X}\right)$, where $i_{\mathfrak g_X}:\mathfrak g_X \to \mathfrak g_X \oplus E$ is the natural inclusion $\left(x,\alpha\right)\mapsto \left(\left(x,\alpha\right),0_x\right)$, and maps a morphism $\psi:E\to E'$ to the morphism $\mathrm{id}_\mathfrak g \oplus \psi:\left(\mathfrak g_X\oplus E,i_{\mathfrak g_X}\right) \to \left(\mathfrak g_X\oplus E',i_{\mathfrak g_X}\right)$.
\end{definition}

\begin{remark}
Note that $\mathrm F$ can be written in terms of coproducts in $\mathrm{Vect}_G\left(X\right)$ or products in $\mathfrak g_X \backslash \mathrm{Vect}_G\left(X\right)$ (see  Example \ref{ex: products} above and Proposition \ref{prop: products} below):
\[
\mathrm F = \left(\mathfrak g_X \oplus -,i_{\mathfrak g_X}\right)
= \left(\left(\mathfrak g_X,\mathrm{id}\right) \times -\right) \circ \sigma
\]
\end{remark}

\begin{remark}
\label{rem: UforTG}
Via the isomorphism of Theorem \ref{thm: A}, the functors $\mathrm U,\sigma$ and $\mathrm F$ correspond to functors between $\mathrm{Vect}_{TG}^{\mathrm{aff}}\left(X\right)$ and $\mathrm{Vect}_G\left(X\right)$, which we denote by the same symbols. The functor
\[
\mathrm U: \mathrm{Vect}_{TG}^{\mathrm{aff}}\left(X\right) \to \mathrm{Vect}_G\left(X\right)
\] 
is given by the pullback of actions along the Lie group morphism $0_{TG}:G \to TG$. It maps an affine action $\mu:TG\times E \to E$ to the $G$-vector bundle $E$ with action $g \cdot e_x = 0_g \cdot e_x$. The functor
\[
\sigma: \mathrm{Vect}_G\left(X\right) \to \mathrm{Vect}_{TG}^{\mathrm{aff}}\left(X\right)
\] 
is given by the pullback of actions along $\pi_{TG}:TG\to G$. It maps a $G$-vector bundle $E$ to the affine action defined by $v_g \cdot e_x = g \cdot e_x$. See also Example \ref{example: equivariant bundles} above. The functor
\[
\mathrm F:\mathrm{Vect}_G\left(X\right) \to \mathrm{Vect}_{TG}^{\mathrm{aff}}\left(X\right)
\] 
maps a $G$-equivariant vector bundle $E$ to the affine action $TG\times \left(\mathfrak g_X \oplus E\right) \to \mathfrak g_X \oplus E$ given by 
\[
v_g \cdot \left(\left(x,\alpha\right),e_x\right)  = \left(\left(g\cdot x,\alpha+\left(R_{g^{-1}}\right)_\ast v_g\right),g\cdot e_x\right).
\]
\end{remark}

\subsection{Lemma on slice categories}

This subsection contains the proof of the following category-theoretic result. Lemma \ref{lemma: D} will only be used to prove Theorem \ref{thm: C} below and will not be referred to elsewhere, so that the reader uninterested in category-theoretic abstractions may wish to skip to section \ref{subsec: thm B} below. 

Recall that if $\mathcal U : \mathscr D \rightleftarrows \mathscr C : \mathcal F$ are a pair of adjoint functors, with $\mathcal F$ left adjoint to $\mathcal U$, then there is an associated \emph{monad} acting on the category $\mathscr C$. This monad consists of the endofunctor $\mathbb T = \mathcal U\mathcal F : \mathscr C \to \mathscr C$, together with a certain pair $\chi,\eta$ of natural transformations which are constructed from the unit and counit of the adjunction. There is a canonical \emph{comparison functor} $\mathcal K:  \mathscr D \to \mathscr C^{\left<\mathbb T,\chi,\eta\right>}$, from $\mathscr D$ to the category of $\left<\mathbb T,\chi,\eta\right>$\emph{-algebras}, and the adjunction is called \emph{monadic} if $\mathcal K$ is an isomorphism. One can therefore understand the statement that an adjunction is monadic as saying that the category $\mathscr D$ can be reconstructed from $\mathscr C$ and the monad $\left<\mathbb T,\chi,\eta\right>$ in a canonical way, which shows the importance of the notion. See \cite{maclane} for further details. 

\begingroup
\setcounter{mainlemma}{3} 
\renewcommand\themainlemma{\Alph{mainlemma}}
\begin{mainlemma}
\label{lemma: D}
Let $\mathscr C$ be a category, $m$ an object in $\mathscr C$, and $\mathcal U: m\backslash \mathscr C \to \mathscr C$ the standard forgetful functor. If the coproduct $m \amalg a$ exists in $\mathscr C$ for all objects $a$ in $\mathscr C$ then the functor $\mathcal F: a \mapsto \left(m\amalg a, i_m\right)$ is left adjoint to $\mathcal U$, and this adjunction is monadic.
\end{mainlemma}
\endgroup

\begin{proof}
Let $\left(b,\phi\right)$ be an object in $m\backslash \mathscr C$. Via the universal property of the coproduct $m\amalg a$ the natural bijection
\begin{align*}
\mathrm{Hom}_\mathscr C\left(m,a'\right) \times \mathrm{Hom}_\mathscr C\left(a,a'\right) & \mathrel{\mathop{\to}^{\cong}_{\Sigma}}
\mathrm{Hom}_\mathscr C\left(m\amalg a,a'\right) 
\end{align*}
restricts to a natural bijection
\begin{align}
\label{eqn: hom-set adj1}
\mathrm{Hom}_\mathscr C\left(a,a'\right)
& 
\xrightarrow{\cong}
\mathrm{Hom}_{m\backslash \mathscr C}
\left( \left(m\amalg a, i_m\right) , \left(a',\phi'\right) \right)
\\
\nonumber
f & \mapsto \Sigma\left(\left(\phi',f\right)\right)
\end{align}
The bijection \eqref{eqn: hom-set adj1} can be rewritten as
\begin{align}
\label{eqn: homset adj}
\mathrm{Hom}_\mathscr C\left(a,\mathcal U\left(\left(a',\phi\right)\right)\right)
& \xrightarrow{\cong}
\mathrm{Hom}_{m\backslash \mathscr C}
\left( \mathcal F\left(a\right) , \left(a',\phi'\right) \right)
\end{align}
This shows that $\mathcal F$ is left adjoint to $\mathcal U$. 

It remains to show that the adjunction $\mathcal F \dashv \mathcal U$ defined by \eqref{eqn: homset adj} is monadic. We must show that the canonical comparison functor $\mathcal K: m\backslash \mathscr C \to \mathscr C^{\left< \mathbb T,\chi,\eta\right>}$ from $m\backslash \mathscr C$ to the category of $\left< \mathbb T,\chi,\eta\right>$-algebras is an isomorphism, where $\left< \mathbb T,\chi,\eta\right>$ is the monad associated to the adjunction, and $\mathbb T,\chi,\eta$ and $\mathcal K$ are defined below (see \cite{maclane} for the general case). 

Straightforward calculations show that the unit $\eta:1_{\mathscr C}\Rightarrow \mathcal U \mathcal F$ and counit $\varepsilon:\mathcal F \mathcal U \Rightarrow 1_{m\backslash \mathscr C}$ of the adjunction have components 
\begin{align*}
\eta_a = i_a : a & \to m \amalg a \\
\varepsilon_{\left(a',\phi\right)} = \left(\phi,1_{a'}\right) : \left(m \amalg a',i_m\right) & \to \left(a',\phi\right) 
\end{align*}
We denote the associated monad by $\left< \mathbb T,\chi,\eta\right>$, where $\mathbb T = \mathcal U \mathcal F : \mathscr C \to \mathscr C$, $\chi=\mathcal U \varepsilon \mathcal F : \mathbb T^2 \Rightarrow \mathbb T$, and $\eta:1_\mathscr C \Rightarrow \mathbb T$. 
The functor $\mathbb T$ is the functor $a \mapsto m\amalg a$. A further calculation shows that $\chi$ has components
\begin{align*}
\chi_{a} = \left(i_m, 1_{m\amalg a}\right) : m \amalg \left(m \amalg a\right) & \to m \amalg a
\end{align*}

We can now describe the category $\mathscr C^{\left<\mathbb T,\chi,\eta\right>}$ of $\left<\mathbb T,\chi,\eta\right>$-algebras, or just $\mathbb T$-algebras for short. Let $\left<a,h\right>$ be a $\mathbb T$-algebra. Then $h$ is a morphism $h: \mathbb Ta = m\amalg a \to a$ in $\mathscr C$ such that $h \circ \eta_a = 1_a$ as morphisms $a \to a$, and $h \circ \mathbb Th = h \circ \chi_a$ as morphisms $\mathbb T^2 a \to a$. Under the natural bijection
\begin{align*}
\mathrm{Hom}_\mathscr C\left(m\amalg a,a\right) 
& \xrightarrow{\cong} 
\mathrm{Hom}_\mathscr C\left(m,a\right) \times \mathrm{Hom}_\mathscr C\left(a,a\right) 
\end{align*}
the first condition on $h$ is equivalent to the condition that $h$ is of the form $\left(\bar h,1_a\right)$. A calculation then shows that the second condition on $h$ holds automatically. It follows that $\mathbb T$-algebras can be identified with objects of $m \backslash \mathscr C$. 

Let $f:\left<a,h\right> \to \left<a',h'\right>$ be a morphism of $\mathbb T$-algebras. Then $f:a \to a'$ is a morphism in $\mathscr C$ such that $h' \circ \mathbb Tf = f \circ h$ as morphisms $\mathbb Ta \to a'$. Identifying $h=\left(\bar h,1_a\right)$ and $h'=\left(\bar h',1_{a'}\right)$ as above, the condition on $f$ is
\[
\left(\bar h',1_a\right) \circ \left(1_m \amalg f\right)
=
f \circ \left(\bar h,1_a\right)
\]
which is equivalent to
\[
\left(\bar h',f\right) = \left(f\bar h,f\right)
\]
and therefore to $\bar h' = f\bar h$. It follows that morphisms between $\mathbb T$-algebras can be identified with morphisms in $m\backslash \mathscr C$. The discussion above shows that there is an isomorphism of categories $m\backslash \mathscr C \xrightarrow{\cong}  \mathscr C^{\left<\mathbb T,\chi,\eta\right>} $ defined on objects by
\begin{align*}
\left(a,\bar h\right) & \mapsto  \left<a,\left(\bar h,1_a\right)\right>  \\
& = \left<\mathcal U\left(a,\bar h\right), \mathcal U\left(\varepsilon_{\left(a,\bar h\right)}\right)\right> 
\end{align*}
and on morphisms by
\begin{align*}
\left(f: 
\left(a,\bar h\right)
\to 
\left(a',\bar h'\right)
\right)
& \mapsto
\left(f: 
\left<a,\left(\bar h,1_{a}\right)\right>
\to 
\left<a',\left(\bar h',1_{a'}\right)\right>
\right)  \\
& = \bigg(\mathcal Uf : 
\left<\mathcal U\left(a,\bar h\right), \mathcal U\left(\varepsilon_{\left(a,\bar h\right)}\right)\right>
\\
& \;\;\;\;\; \to
\left<\mathcal U\left(a',\bar h'\right), \mathcal U\left(\varepsilon_{\left(a',\bar h'\right)}\right)\right>\bigg)
\end{align*}
This functor is exactly the comparison functor $\mathcal K: m\backslash \mathscr C \to \mathscr C^{\left<\mathbb T,\chi,\eta\right>}$ (see \cite{maclane}), which implies that the adjunction $\mathcal F \dashv \mathcal U$ is monadic.
\end{proof}

\subsection{Theorem B}
\label{subsec: thm B}

\begingroup
\setcounter{mainthm}{1} 
\renewcommand\themainthm{\Alph{mainthm}}
\begin{mainthm}
\label{thm: B}
The following statements hold:
\begin{enumerate}
\item $\mathrm F$ is left adjoint to $\mathrm U$.
\item The adjunction $\mathrm F \dashv \mathrm U$  is monadic.
\item $\sigma$ is the unique section of $\mathrm U$.
\end{enumerate}
\end{mainthm}
\endgroup

\begin{proof}
The category $\mathrm{Vect}_G\left(X\right)$ is additive and so, in particular, has finite coproducts. If $E$ is an object in $\mathrm{Vect}_G\left(X\right)$ then the coproduct $\mathfrak g_X \amalg E$ is given by the $G$-vector bundle $\mathfrak g_X \oplus E$ together with the natural inclusions $i_{\mathfrak g_X}:\mathfrak g_X \to \mathfrak g_X\oplus E$ and $i_E:E \to \mathfrak g_X\oplus E$. It follows that $\mathrm F$ corresponds to the functor $\mathcal F: a \mapsto \left(c\amalg a, i_c\right)$ in the statement of Lemma \ref{lemma: D}, and $\mathrm U$ to the functor $\mathcal U$. The first two statements in the Theorem therefore follow from Lemma \ref{lemma: D}. 

We will prove the third statement. If $E$ is an object in $\mathrm{Vect}_G\left(X\right)$ then $\mathrm U\left(\sigma \left(E\right)\right) = \mathrm U \left(E,0\right) = E$, and if $\psi:E\to E'$ is a morphism then $\mathrm U\left( \sigma\left(\psi\right)\right) = \mathrm U \left( \psi:\left(E,0\right) \to \left(E',0\right) \right) = \psi:E\to E'$, so that  $\sigma$ is a section of $\mathrm U$. It remains to show that it is unique. Let $\sigma'$ be a section of $\mathrm U$. If $E$ is an object in $\mathrm{Vect}_G\left(X\right)$ then $\sigma'\left(E\right) = \left(E,\phi\right)$ for some morphism $\phi:\mathfrak g_X \to E$. Applying $\sigma$ to the zero-morphism $0:E \to E$ produces a morphism $0: \left(E,\phi\right) \to \left(E,\phi\right)$ in $\mathfrak g_X \backslash \mathrm{Vect}_G\left(X\right)$ which implies that $\phi = 0 \circ \phi = 0$. Therefore $\sigma'\left(E\right)=\left(E,0\right)=\sigma\left(E\right)$. If $\psi:E\to E'$ is a morphism in $\mathrm{Vect}_G\left(X\right)$ then the fact that $\sigma'$ is a section of $\mathrm U$ implies that $\sigma'\left(\psi\right) = \psi :\left(E,0\right) \to \left(E',0\right)$. Therefore $\sigma'=\sigma$.
\end{proof}

\begin{remark}
\label{remark: homset adj}
It follows from the details of the proof of Lemma \ref{lemma: D}, and in particular \eqref{eqn: homset adj}, that the adjunction $\mathrm F\dashv \mathrm U$ of Theorem \ref{thm: B} is given in terms of hom-sets by:
\begin{align*}
\mathrm{Hom}_{\mathrm{Vect}_G\left(X\right)}\left(E,\mathrm{U}\left(\left(E',\phi'\right)\right)\right)
& \xrightarrow{\cong}
\mathrm{Hom}_{\mathfrak g_X \backslash \mathrm{Vect}_G\left(X\right)}
\left(\mathrm{F}\left(E\right) , \left(E',\phi'\right) \right)
\end{align*}
\[
\xymatrix{
E \ar[d]^f  &   \ar@{|->}[r]  & & \mathfrak g_X \ar[dr]_{\phi'} \ar[r]^-{i_{\mathfrak g_X}}& \mathfrak g_X \oplus E \ar[d]^{\left(\phi',f\right)} \\
E' &&&& E'
}
\]
\end{remark}

\begin{remark}
\label{rem: analogue thmB}
It follows from the isomorphism $\mathrm{Vect}_{TG}^{\mathrm{aff}}\left(X\right) \cong \mathfrak g_X \backslash \mathrm{Vect}_G\left(X\right)$ of Theorem \ref{thm: A}, that the analogue of Theorem \ref{thm: B} involving the functors $\mathrm U: \mathrm{Vect}_{TG}^{\mathrm{aff}}\left(X\right) \to \mathfrak g_X \backslash \mathrm{Vect}_G\left(X\right)$ and $\mathrm F,\sigma : \mathrm{Vect}_G\left(X\right) \to \mathrm{Vect}_{TG}^{\mathrm{aff}}\left(X\right)$ defined in Remark \ref{rem: UforTG} also holds.
\end{remark}


\section{Pullbacks, products and K-theory}
\label{sec: K}

In this section we first show that a $G$-map $f:X\to Y$ determines a pullback functor $\mathrm{Vect}_{TG}^{\mathrm{aff}}\left(Y\right) \to \mathrm{Vect}_{TG}^{\mathrm{aff}}\left(X\right)$. Unlike $\mathrm{Vect}_G\left(X\right)$, $\mathrm{Vect}_{TG}^{\mathrm{aff}}\left(X\right)$ does not carry any obvious additive structure. Nonetheless, $\mathrm{Vect}_{TG}^{\mathrm{aff}}\left(X\right)$ does have initial and terminal objects, and finite products. This will allow us to define the Grothendieck group $K_{TG}^{\mathrm{aff}}\left(X\right)$, and to show that that the above mentioned pullback functors determine group homomorphisms between these groups. Our third main result, Theorem \ref{thm: C}, shows that the functors $\mathrm U$ and $\sigma$ defined in section \ref{sec: functors} determine a natural isomorphism between $K_{TG}^{\mathrm{aff}}\left(X\right)$ and the equivariant K-theory $KO_G\left(X\right)$.

\subsection{Pullback functors}
\label{subsec: pullback functors}

If $f_1:X_1 \to X_2$ is a $G$-equivariant smooth map then the pullback of $G$-vector bundles determines a functor $f_1^\ast : \mathrm{Vect}_G\left(X_2\right) \to \mathrm{Vect}_G\left(X_1\right)$. If $f_2:X_2 \to X_3$ is a composable $G$-equivariant smooth map then there is a natural isomorphism of functors $\left(f_2f_1\right)^\ast \cong f_1^\ast f_2^\ast$. Note that there is a canonical isomorphism of $G$-vector bundles $f_1^\ast \mathfrak g_{X_2} \cong \mathfrak g_{X_1}$ given by $\left(x,\left(f_1(x),\xi\right)\right) \mapsto \left(x,\xi\right)$. 

\begin{definition}
\label{def: pullback}
If $f:X \to Y$ is a $G$-equivariant smooth map then we define the functor
\[
 \tilde f^\ast : \mathfrak g_Y \backslash \mathrm{Vect}_G\left(Y\right)
\to
 \mathfrak g_X \backslash \mathrm{Vect}_G\left(X\right)
\]
to be the composition
\begin{equation}
\label{eqn: pullback defn}
\mathfrak g_X \backslash \mathrm{Vect}_G\left(Y\right) \to
\left(f^\ast \mathfrak g_Y\right) \backslash \mathrm{Vect}_G\left(X\right)
\xrightarrow{\cong}
\mathfrak g_X \backslash \mathrm{Vect}_G\left(X\right)
\end{equation}
where the first functor in \eqref{eqn: pullback defn} is the functor between slice categories determined by the pullback of $G$-vector bundles, and the second is determined by the canonical isomorphism $f^\ast \mathfrak g_Y \cong \mathfrak g_X$.
\end{definition}

\begin{remark}
\label{rem: TGpullback}
In terms of affine actions the action of $TG$ on $f^\ast E$ is given by $v_g \cdot \left(x,e\right) = \left(g\cdot x,v_g \cdot e\right)$.
\end{remark}

The following Proposition follows immediately from the corresponding result for $G$-vector bundles:

\begin{proposition}
\label{prop: natural isom of pullbacks}
If $f_1,f_2$ are composable $G$-equivariant maps and $\tilde f_1^\ast, \tilde f_2^\ast$ are the pullback functors defined in Definition \ref{def: pullback}, then there is a natural isomorphism $\widetilde{\left(f_2 f_1\right)}^\ast \cong \tilde f_1^\ast \tilde f_2^\ast$.
\end{proposition}

\begin{remark}
The categories $\mathrm{Vect}_{TG}^{\mathrm{aff}}\left(X\right)$ for varying $X$, functors of Definition \ref{def: pullback}, and natural isomorphisms of Proposition \ref{prop: natural isom of pullbacks} in fact constitute a \emph{pseudo-functor} $G\text{-Man} \to \mathrm{Cat}$ from the category of $G$-manifolds to the 2-category  of (essentially small) categories, but we shall not make use of this fact.
\end{remark}

The pullback functors defined above are compatible with the functors $\mathrm U,\sigma$ and $\mathrm F$ defined in section \ref{subsec: forget} in the following sense. Let us temporarilly record the dependence on $X$ by $\mathrm U_X,\sigma_X$ and $\mathrm F_X$. The following Proposition then follows from the definitions of these functors:
\begin{proposition}
\label{prop: naturality of functors}
Let $f:X \to Y$ be a $G$-map. Then there are natural isomorphisms 
\[
\mathrm U_X \tilde f^\ast \cong f^\ast \mathrm U_Y, \; \;
\sigma_X f^\ast \cong \tilde f^\ast \sigma_Y, \; \;
\mathrm F_X f^\ast \cong \tilde f^\ast \mathrm F_Y.
\]
\end{proposition}

\subsection{Products}
\label{subsec: products}

\begin{proposition}
\label{prop: terminal}
$\mathfrak g_X \backslash\mathrm{Vect}_G\left(X\right)$ has a terminal object, which is given by $\left(X \times \left\{0\right\},0\right)$, and an initial object given by $\left(\mathfrak g_X,\mathrm{id}\right)$. 
\end{proposition}

\begin{proof}
This follows from the definition of the slice category $\mathfrak g_X \backslash \mathrm{Vect}_G\left(X\right)$ and the fact that $X\times \left\{0\right\}$ is a terminal object in $\mathrm{Vect}_G\left(X\right)$. 
\end{proof}

\begin{corollary}
\label{cor: TGterminal}
The category  $\mathrm{Vect}_{TG}^{\mathrm{aff}}\left(X\right)$ has a terminal object given by the zero-vector bundle $X\times \left\{0\right\}$ with action $v_g \cdot \left(x,0\right) = \left(g \cdot x, 0\right)$, and an initial object given by the vector bundle $\mathfrak g_X$ with action $v_g \cdot \left(x,\xi\right) = \left(g \cdot x, \xi + \left(R_{g^{-1}}\right)_\ast v_g\right)$.
\end{corollary}

\begin{proof}
This follows from the isomorphism of Theorem \ref{thm: A}.
\end{proof}

\begin{proposition}
\label{prop: products}
The category $\mathfrak g_X\backslash \mathrm{Vect}_G\left(X\right)$ has finite products. The product of $\left(E,\phi\right)$ and $\left(E',\phi'\right)$ is given by the object $\left(E\oplus E',\left(\phi,\phi'\right)\right)$ and the canonical projections to $\left(E,\phi\right)$ and $\left(E',\phi'\right)$.
\end{proposition}

 \begin{proof}
The result follows from general facts about limits in slice categories. Explicitly, let $\left(F,\psi\right)$ be an object in $\mathfrak g_X \backslash \mathrm{Vect}_G\left(X\right)$ and $f:\left(F,\psi\right) \to \left(E,\phi\right)$ and $f':\left(F,\psi\right) \to \left(E',\phi'\right)$ be morphisms. In particular, $f\psi =\phi$ and $f'\psi=\phi'$. As $E\oplus E'$ is a product in $\mathrm{Vect}_G\left(X\right)$ there is a unique morphism $\left(f,f'\right):F\to E\oplus E'$ satisfying $\mathrm{pr}_E \circ \left(f,f'\right) = f$ and $\mathrm{pr}_{E'}\circ \left(f,f'\right) = f'$. We have $\left(f,f'\right)\circ \psi = \left(f\psi,f'\psi\right) = \left(\phi,\phi'\right)$ and therefore $\left(f,f'\right)$ defines a morphism $\left(f,f'\right):\left(F,\psi\right) \to \left(E\oplus E',\left(f,f'\right)\right)$ in $\mathfrak g_X\backslash \mathrm{Vect}_G\left(X\right)$. Finally, the identities above involving $\mathrm{pr}_E$ and $\mathrm{pr}_{E'}$ imply that the same identities hold when these morphisms are considered as morphisms in $\mathrm{Vect}_{TG}^{\mathrm{aff}}\left(X\right)$. This shows that $\left(E\oplus E',\left(\phi,\phi'\right)\right)$ and the canonical projections to $\left(E,\phi\right)$ and $\left(E',\phi'\right)$ satisfy the required universal property.
\end{proof}

\begin{corollary}
\label{cor: TGproduct}
The category $\mathrm{Vect}_{TG}^{\mathrm{aff}}\left(X\right)$ has finite products. If $E$ and $E'$ are vector bundles equipped with affine actions then their product is the vector bundle $E\oplus E'$ with affine action $v_g \cdot \left(e,e'\right) = \left(v_g \cdot e,v_g \cdot e'\right)$.
\end{corollary}

\begin{proof}
This follows from the isomorphism of Theorem \ref{thm: A}.
\end{proof}

\begin{proposition}
The functors $\mathrm U$ and $\sigma$ preserve products.
\end{proposition}

\begin{proof}
This follows immediately from the definitions of the functors $\mathrm U$ and $\sigma$ and the description of products in $\mathfrak g_X \backslash \mathrm{Vect}_G\left(X\right)$ given in Proposition \ref{prop: products}.
\end{proof}

\begin{remark}
Note that if $G$ is not discrete then the left adjoint functor $\mathrm F$ of Definition \ref{def: left adjoint} does not preserve products: 
\begin{align*}
\mathrm F \left(E\right) \times \mathrm F\left(E'\right) 
& =
\left(\left(\mathfrak g_X\oplus E\right)\oplus\left(\mathfrak g_X\oplus E'\right),\left(i_{\mathfrak g_X},i_{\mathfrak g_X}\right)\right) \\
& \ncong
\left(\mathfrak g_X\oplus \left(E\oplus E'\right),i_{\mathfrak g_X}\right) \\
& = \mathrm F\left(E\times E'\right)
\end{align*}
\end{remark}

\begin{proposition}
\label{prop: tilde f products}
If $f:X\to Y$ is a $G$-map then the functor $\tilde f^\ast : \mathfrak g_Y \backslash \mathrm{Vect}_G\left(Y\right)
\to
 \mathfrak g_X \backslash \mathrm{Vect}_G\left(X\right)$ of Definition \ref{def: pullback} preserves products.
\end{proposition}

\begin{proof}
This follows from the definition of the functor $\tilde f^\ast$ and the fact that the functor $f^\ast : \mathrm{Vect}_G\left(Y\right) \to \mathrm{Vect}_G\left(X\right)$ preserves products.
\end{proof}

\subsection{K-theory}

\begin{definition}
\label{def: KTG}
Let $K_{TG}^{\mathrm{aff}}\left(X\right)$ be the Grothendieck group of the category $\mathrm{Vect}_{TG}^{\mathrm{aff}}\left(X\right)$.
\end{definition}

It follows from Proposition \ref{prop: tilde f products} and the functoriality of the construction of the Grothendieck group (see section \ref{subsec: Grothendieck}) that the functor $\tilde f^\ast : \mathfrak g_Y \backslash \mathrm{Vect}_G\left(Y\right)
\to \mathfrak g_X \backslash \mathrm{Vect}_G\left(X\right)$ associated to a $G$-map $f:X\to Y$ induces a group homomorphism $K\left(\tilde f^\ast\right):K_{TG}^{\mathrm{aff}}\left(Y\right) \to K_{TG}^{\mathrm{aff}}\left(X\right)$. 

\begin{definition}
We denote by
\[
K_{TG}^{\mathrm{aff}}\left(-\right) : G\text{-Man} \to \mathrm{Ab}
\]
the functor which maps a $G$-manifold $X$ to the abelian group $K_{TG}^{\mathrm{aff}}\left(X\right)$ and maps a $G$-map $f:X\to Y$ to the group homomorphism $K\left(\tilde f^\ast\right)$.
\end{definition}

Our last main result is that the functor $K_{TG}^{\mathrm{aff}}\left(-\right)$ is isomorphic to $KO_G\left(-\right)$ (recall from section \ref{subsec: Grothendieck} that $KO_G\left(X\right)$ denotes the Grothendieck group of $\mathrm{Vect}_G\left(X\right)$). The proof is based on the following Lemma, which states the category-theoretic reasons for this result.

\begingroup
\setcounter{mainlemma}{4} 
\renewcommand\themainlemma{\Alph{mainlemma}}
\begin{mainlemma}
\label{lemma: E}
Let $\mathscr C$ be an additive category and $m$ an object in $\mathscr C$. Let $\mathcal U: m\backslash \mathscr C \to \mathscr C$ be the standard forgetful functor and $\mathcal S: \mathscr C \to m \backslash \mathscr C$ the section $a \mapsto \left(a,0\right)$. Then the group homomorphism
\[
K\left(\mathcal U \right): K\left(m\backslash \mathscr C\right) \to K\left(\mathscr C\right) 
\]
is an isomorphism. Its inverse is
\[
K\left(\mathcal S\right) : 
K\left(\mathscr C\right) \to
 K\left(m\backslash \mathscr C\right) 
\]
\end{mainlemma}
\endgroup

\begin{proof}
As $\mathcal U \circ \mathcal S = 1_\mathscr C$, the functoriality of the Grothendieck group implies that the homomorphism $K\left(\mathcal U\right)$ is surjective, $K\left(\mathcal S\right)$ is injective, and $K\left(\mathcal U\right) \circ K\left(\mathcal S\right) = 1_{K\left(\mathscr C\right)}$. The Lemma will therefore follow from the fact that $K\left(\mathcal S\right)$ is surjective, which we will prove. 

Let $\left(a,f\right)$ be an object in $m\backslash \mathscr C$. Consider the following diagram:
\begin{equation}
\label{diag: prodtriangle}
\xymatrix{
m \ar[r]^-{\left(f,f\right)} \ar[dr]_{\left(f,0\right)} &
a \oplus a \ar[d]^{h} \\
& a \oplus a
}
\end{equation}
where $h = \left(\mathrm{pr}_1,\mathrm{pr}_2 - \mathrm{pr}_1\right)$, or in terms of elements $h\left(e_1,e_2\right) = \left(e_1, e_2 - e_1\right)$. It is clear that \eqref{diag: prodtriangle} commutes, and $h$ is an isomorphism with inverse $\left(\mathrm{pr}_1,\mathrm{pr_2} + \mathrm{pr}_1\right)$. Therefore, $h$ determines an isomorphism
\[
\left(a,f\right) \times \left(a,f\right) \xrightarrow{\cong}
\left(a,f\right) \times \left(a,0\right)
\]
in $m\backslash \mathscr C$. It follows that in the Grothendieck group $K\left(m\backslash \mathscr C\right)$ we have
\[
\left[ \left(a,f\right) \right]  = \left[ \left(a,0\right) \right] = K\left(\mathcal S\right) \left(\left[a\right] \right)
\]
Therefore, if $\left(a_1,f_1\right)$ and $\left(a_2,f_2\right)$ are objects in $m\backslash \mathscr C$, then
\[
\left[ \left(a_1,f_1\right) \right] 
-
\left[ \left(a_2,f_2\right) \right]
=
K\left(\mathcal S\right)
\left(
\left[a_1\right] - \left[a_2\right] 
\right)
\]
which shows that $K\left(\mathcal S\right)$ is surjective.
\end{proof}

\begin{remark}
There exist similar results describing the Grothendieck groups of various categories associated to an additive category $\mathscr C$. For example, the Main Theorem in \cite{almkvist} describes the group $K\left(\mathrm{end} \; P\left(A\right)\right)$, where $\mathrm{end} \; P\left(A\right)$ is the additive category of endomorphisms of finitely generated projected modules over a commutative ring $A$.
\end{remark}

\begingroup
\setcounter{mainthm}{2} 
\renewcommand\themainthm{\Alph{mainthm}}
\begin{mainthm}
\label{thm: C}
If $X$ is a $G$-manifold then the functor $\mathrm U$ induces a group isomorphism
\[
K\left(\mathrm U\right):K_{TG}^{\mathrm{aff}}\left(X\right) \to KO_G\left(X\right)
\]
Its inverse is 
\[
K\left(\sigma\right) : KO_G\left(X\right) \to K_{TG}^{\mathrm{aff}}\left(X\right)
\]
These isomorphisms are natural in $X$, and thus define an isomorphism of functors
\[
K_{TG}^{\mathrm{aff}}\left(-\right) \xrightarrow{\cong} KO_G\left(-\right)
\]
\end{mainthm}
\endgroup
\begin{proof}
For fixed $X$, the fact that $K\left(\mathrm U\right)$ and $K\left(\sigma\right)$ are mutually inverse isomorphisms follows from Lemma \ref{lemma: E}. The fact that these isomorphisms are natural in $X$ follows from the first two natural isomorphisms in Proposition \ref{prop: naturality of functors} and the functoriality of the Grothendieck group.
\end{proof}

\begin{remark}
\label{rem: monoids vs groups}
Note that Theorem \ref{thm: C} does not hold at the level of monoids. The functor $\mathrm U$ induces a morphism of commutative monoids
\[
\left(\mathfrak g_X \backslash \mathrm{Vect}_G\left(X\right)\right) / \cong \; \;  
\to 
\left(\mathrm{Vect}_G\left(X\right)\right) / \cong
\] 
which is an isomorphism if and only if $G$ is discrete. In particular, $\left[\left(\mathfrak g_X,0\right)\right]$ and $\left[\left(\mathfrak g_X,\mathrm{id}\right)\right]$ are both mapped to $\left[\mathfrak g_X\right]$. This can be seen most explicitly in the situation of Proposition \ref{prop: K(point)} (see also Remark \ref{rem: K(point)}). In this case the equality $\left[\left(\mathfrak g,\mathrm{id}\right)\right] = \left[\left(\mathfrak g,0\right)\right]$ in $K_{TG}^{\mathrm{aff}}\left(\ast\right)$ implies the equality
\[
\left[\left(\mathfrak g,\mathrm{id}\right)^n \times \left(W,0\right)\right] = \left[\left(\mathfrak g,0\right)^n \times \left(W,0\right)\right]
\]
in $K_{TG}^{\mathrm{aff}}\left(\ast\right)$ between the two classes of elements appearing in the classification of objects in $\mathfrak g \backslash \mathrm{Rep}\left(G\right)$ given in Proposition \ref{prop: K(point)}.
\end{remark}

\section{The complex case}
\label{sec: complex}

It is possible reformulate the notion of affine actions in the complex setting. The complexified tangent bundle $T_\mathbb C G$ is both a complex vector bundle and a \emph{real} Lie group. We define an affine action of $T_\mathbb CG$ on a complex vector bundle $E$ to be an action for which the action map $T_\mathbb CG \times E \to E$ is a morphism of \emph{complex} vector bundles. For example, if $X$ is a $G$-manifold with action $G\times X \to X$ then the derivative defines a complex affine action $T_\mathbb C G\times T_\mathbb C X \to T_\mathbb CX$ of $T_\mathbb CG$ on the complexified tangent bundle $T_\mathbb CX$ of $X$.

We use $\mathrm{Vect}_{T_\mathbb C G}^{\mathrm{aff},\mathbb C}\left(X\right)$ (respectively $\mathrm{Vect}_G^{\mathbb C}\left(X\right)$) to denote the category of complex affine actions (respectively $G$-equivariant complex vector bundles) over $X$, and $K_{T_\mathbb C G}^{\mathrm{aff},\mathbb C}\left(X\right)$ (respectively $K_G\left(X\right)$) to denote the \\ Grothendieck group of $\mathrm{Vect}_{T_\mathbb C G}^{\mathrm{aff},\mathbb C}\left(X\right)$ (respectively $\mathrm{Vect}_G^{\mathbb C}\left(X\right)$). As standard, $\mathfrak g_\mathbb C$ denotes the complexified Lie algebra of $G$ equipped with the complexified adjoint representation. 

The functors $\mathrm U,\sigma$ and $\mathrm F$ defined in section \ref{subsec: forget} have analogues in the complex case, which we denote by 
\[
\mathrm U^\mathbb C : \mathrm{Vect}_{T_\mathbb CG}^{\mathrm{aff},\mathbb C}\left(X\right) \to \mathrm{Vect}_G^{\mathbb C}\left(X\right)
\] 
and 
\[
\mathrm F^\mathbb C,\sigma^\mathbb C: \mathrm{Vect}_G^{\mathbb C}\left(X\right) \to \mathrm{Vect}_{T_\mathbb CG}^{\mathrm{aff},\mathbb C}\left(X\right)
\]

\subsection{Complex theorems}

The following analogues of Theorems \ref{thm: A},\ref{thm: B} and \ref{thm: C} hold with essentially the same proofs.

\begingroup
\setcounter{mainthm}{0} 
\renewcommand\themainthm{\Alph{mainthm}$'$}
\begin{mainthm}
\label{thm: A'}
The following three categories are isomorphic:
\begin{enumerate}
\item The category $\mathrm{Vect}_{T_\mathbb C G}^{\mathrm{aff},\mathbb C}\left(X\right)$ of complex affine actions of $T_\mathbb CG$ over $X$.
\item The category of pairs $\left(E,\rho_\mathfrak g\right)$, defined as follows: 
\begin{itemize}
\item The objects are pairs $\left(E,\rho_\mathfrak g\right)$, where $E$ is a $G$-equivariant complex vector bundle over $X$ and $\rho_{\mathfrak g_\mathbb C} : \mathfrak g_\mathbb C \to \Gamma\left(E\right)$ is a $G$-equivariant complex linear map. 
\item The morphisms $\left(E,\rho_\mathfrak g\right) \to \left(E',\rho_\mathfrak g\right)$ are morphisms $\psi:E\to E'$ of $G$-equivariant complex vector bundles over $X$ such that $\Gamma\left(\psi\right)\circ \rho_{\mathfrak g_\mathbb C} = \rho'_{\mathfrak g_\mathbb C}$.
\item Composition is given by composition of morphisms of complex vector bundles over $X$.
\end{itemize}
\item The slice category $ \left(\mathfrak g_\mathbb C\right)_X \backslash \mathrm{Vect}_G^\mathbb C\left(X\right)$.
\end{enumerate}
\end{mainthm}
\endgroup

\begingroup
\setcounter{mainthm}{1} 
\renewcommand\themainthm{\Alph{mainthm}$'$}
\begin{mainthm}
\label{thm: B'}
The following statements hold:
\begin{enumerate}
\item $\mathrm F^\mathbb C$ is left adjoint to $\mathrm U^\mathbb C$.
\item The adjunction $\mathrm F^\mathbb C \dashv \mathrm U^\mathbb C$  is monadic.
\item $\sigma^\mathbb C$ is the unique section of $\mathrm U^\mathbb C$.
\end{enumerate}
\end{mainthm}
\endgroup

\begingroup
\setcounter{mainthm}{2} 
\renewcommand\themainthm{\Alph{mainthm}$'$}
\begin{mainthm}
\label{thm: C'}
If $X$ is a $G$-manifold then the functor $\mathrm U^\mathbb C$ induces a group isomorphism
\[
K\left(\mathrm U^\mathbb C\right):K_{T_\mathbb C G}^{\mathrm{aff},\mathbb C}\left(X\right) \to K_G\left(X\right)
\]
Its inverse is 
\[
K\left(\sigma^\mathbb C\right) : K_G\left(X\right) \to K_{T_\mathbb CG}^{\mathrm{aff},\mathbb C}\left(X\right)
\]
These isomorphisms are natural in $X$, and thus define an isomorphism of functors
\[
K_{T_\mathbb CG}^{\mathrm{aff},\mathbb C}\left(-\right) \xrightarrow{\cong} K_G\left(-\right)
\]
\end{mainthm}
\endgroup



James Waldron, School of Mathematics and Statistics, Newcastle
University, Newcastle upon Tyne NE1 7RU, UK.

\textit{Email address}: \href{mailto:james.waldron@newcastle.ac.uk}{james.waldron@newcastle.ac.uk}


\end{document}